\documentclass{article}

\usepackage{amsmath}
\usepackage{amssymb}
\usepackage{amsthm}
\usepackage{amsfonts}
\usepackage{graphicx}
\usepackage{tikz-cd}

\theoremstyle{plain}
\newtheorem{thm}{Theorem}[section]
\newtheorem{lem}[thm]{Lemma}
\newtheorem{prop}[thm]{Proposition}
\newtheorem{cor}[thm]{Corollary}

\newtheorem{conj}[thm]{Conjecture}

\theoremstyle{definition}
\newtheorem{dfn}[thm]{Definition}
\newtheorem{remark}[thm]{Remark}
\newtheorem{example}[thm]{Example}

\newcommand{\Cocone}{\mathrm{Cocone}}
\newcommand{\Cone}{\mathrm{Cone}}
\newcommand{\Exp}{\mathrm{Exp}}
\newcommand{\Ext}{\mathrm{Ext}}
\newcommand{\Hom}{\mathrm{Hom}}
\newcommand{\Mod}{\mathrm{Mod}}
\newcommand{\Sh}{\mathrm{Sh}}

\begin{document}

\title{Guillermou--Kashiwara--Schapira kernels of geodesic flows}
\author{Takumi Arai
\footnote{
arai.takumi.74e(at)st.kyoto-u.ac.jp
}
}
\date{}
\maketitle

\begin{abstract}
Guillermou--Kashiwara--Schapira proved that there exists a unique sheaf
quantization of any homogeneous Hamiltonian isotopy on a cotangent
bundle.
In this paper, we explicitly construct sheaf quantizations of geodesic flows on spheres and complex projective spaces.
\end{abstract}

\section{Introduction}

In microlocal sheaf theory, there is the notion of a microsupport of a sheaf. The microsupport of a sheaf is a closed conic subset of a cotangent bundle. Moreover, it is a Lagrangian submanifold of a cotangent bundle under some conditions. There is a relation between symplectic geometry and microlocal sheaf theory.

Let $M$ be a differentiable manifold, $\dot{T}^*M$ be its cotangent bundle with the zero section removed, and $I$ be an open interval of $\mathbb{R}$ containing $0$ parametrized by time $t$. 
Let $\mathbb{K}$ be a commutative unital ring and $\Sh(M)$ be the dg derived category of sheaves of $\mathbb{K}$-modules on $M$.
Let $\mathbb{K}_Z$ be the constant sheaf supported on a subset $Z$.
Tamarkin introduced a sheaf quantization of a Hamiltonian isotopy in \cite{Tam}. Guillermou--Kashiwara--Schapira \cite{GKS} proved the unique global existence of a sheaf quantization of a homogeneous Hamiltonian isotopy.

\begin{thm}[\cite{GKS}]
Let $\phi \colon  \dot{T}^*M \times I \to \dot{T}^*M$ be a homogeneous Hamiltonian isotopy. Fix $t \in I$ and set $\phi_t :  =\phi(-,t)$. Then there exists an equivalence of categories $\Phi_t \colon  \Sh(M) \to \Sh(M)$ satisfying $\phi_t({ \mathrm{SS} (F) }\setminus 0_M ) = \mathrm{SS} (\Phi_t(F)) \setminus 0_M $ for any $F \in \Sh(M)$.
\end{thm}

More strongly,

\begin{thm}[\cite{GKS}]
Let $\phi \colon  \dot{T}^*M \times I \to \dot{T}^*M $ be a homogeneous Hamiltonian isotopy. Then there exists an object $K_\phi \in \Sh(M \times M \times I)$ satisfying $ \mathrm{SS} (K_\phi) \setminus 0_{M \times M \times I} = \Lambda_\phi $ and $K_\phi |_{\{t=0\}} \cong \mathbb{K}_{\Delta_M}$, where $\Lambda_\phi \subset T^*(M \times M \times I)$ is the Lagrangian submanifold defined by $\phi$. 
Moreover, any such object $K_\phi$ is unique up to unique isomorphism in the homotopy category.
\end{thm}

We call the functor $\Phi_t$ the GKS functor and the object $K_\phi$ the GKS kernel.
The GKS kernel $K_\phi$ is important because it has many applications, including displaceability problems \cite{Tam}, diagonal thickenings \cite{PS}, sheaf-theoretical wrappings \cite{Kuo}, Legendrian isotopies \cite{Zhou}, and more.
One way to construct $K_\phi$ is to use the sheaf composition \cite{Zhang}.
One can explicitly construct $K_\phi$ without sheaf composition in the case where $\phi$ is the geodesic flow on $(\mathbb{R}^n,g_{std})$.

\begin{prop}[\cite{GKS}]
\label{prop1.3}
Let $\phi \colon  \dot{T}^*\mathbb{R}^n \times \mathbb{R} \to \dot{T}^*\mathbb{R}^n$ be the geodesic flow induced by the standard metric. Then there exists a morphism in $\mathrm{Sh}(\mathbb{R}^n \times \mathbb{R}^n \times \mathbb{R})$
\begin{equation*}
\mathbb{K}_{\{(x,y,t) \mid |x-y|\leq-t\}}[-1] \xrightarrow{} \mathbb{K}_{\{ (x,y,t) \mid |x-y|<t\}}[n]
\end{equation*}
such that its mapping cone is the GKS kernel $K_\phi$.
\end{prop}

In this paper, we explicitly construct $K_\phi$ in the case where $\phi$ is the geodesic flow on $(S^n,g_{std})$ or $(\mathbb{CP}^n,g_{FS})$. 
Our main theorems are as follows.

\begin{thm}
\label{thm1.4}
Let $\phi \colon  \dot{T}^*S^n \times \mathbb{R} \to \dot{T}^*S^n$ be the geodesic flow induced by the standard metric. Then there exists a sequence in $\Sh(S^n \times S^n \times \mathbb{R})$
\begin{equation*}
\cdots \xrightarrow{\psi_{-3}} \mathbb{K}_{Z_{-2}}[-2n] \xrightarrow{\psi_{-2}} \mathbb{K}_{Z_{-1}}[-n] \xrightarrow{\psi_{-1}} \mathbb{K}_{Z_0}[0] \xrightarrow{\psi_0} \mathbb{K}_{Z_1}[n+1] \xrightarrow{\psi_1} \mathbb{K}_{Z_2}[2n+1] \xrightarrow{\psi_2} \cdots 
\end{equation*}
such that its iterated mapping cone is the GKS kernel $K_\phi$.
\end{thm}

\begin{thm}
\label{thm1.5}
Let $\phi \colon  \dot{T}^*\mathbb{CP}^n \times \mathbb{R} \to \dot{T}^*\mathbb{CP}^n$ be the geodesic flow induced by the Fubini--Study metric. Then there exists a sequence in $\mathrm{Sh}(\mathbb{CP}^n \times \mathbb{CP}^n \times \mathbb{R})$
\begin{equation*}
\cdots \xrightarrow{\psi_{-3}} \mathbb{K}_{Z_{-2}}[-2n-2] \xrightarrow{\psi_{-2}} \mathbb{K}_{Z_{-1}}[-2] \xrightarrow{\psi_{-1}} \mathbb{K}_{Z_0}[0] \xrightarrow{\psi_0} \mathbb{K}_{Z_1}[2n+1] \xrightarrow{\psi_1} \mathbb{K}_{Z_2}[2n+3] \xrightarrow{\psi_2} \cdots 
\end{equation*}
such that its iterated mapping cone is the GKS kernel $K_\phi$.
\end{thm}

The explicit definitions of the locally closed subsets $Z_i$ and the morphisms $\psi_i$ will be given in Section 3 and Section 4. The outline of the proof of the main theorems is as follows.

\begin{enumerate}
\item Find a locally closed subset $Z_i \subset M \times M \times \mathbb{R}$ such that the sheaf $\mathbb{K}_{Z_i} \in \Sh(M\times M\times \mathbb{R}) $ satisfies the microsupport conditions except at the singular subset of $Z_i$.
\item Find a morphism $\psi_i \in \Hom_{\Sh(M \times M \times \mathbb{R})} (\mathbb{K}_{Z_i} , \mathbb{K}_{Z_{i+1}} ) $ such that $\psi_i$ cancels the extra microsupport at the singular subset of $Z_i$ and $Z_{i+1}$.
\end{enumerate}

As an application of Theorem \ref{thm1.5}, we solve a conjecture posed in \cite{KuwagakiSaito} in the case of $M=\mathbb{CP}^n$. The conjecture states that for any $\mathbb{C}$-constructible sheaf $F$ on a compact complex manifold $M$ there exists an inductive system $\{F_i\}_{i \in I}$ which consists of $\mathbb{C}$-constructible sheaves on $M$ such that $\mathrm{hocolim}_{i \in I} F_i \cong \mathrm{hocolim}_{t \in \mathbb{R}} \Phi_t(F)$.

In Section 2, we review symplectic geometry and microlocal sheaf theory.
In Section 3, we explicitly construct the GKS kernel of the geodesic flow on $S^n$.
In Section 4, we explicitly construct the GKS kernel of the geodesic flow on $\mathbb{CP}^n$.

\subsection*{Acknowledgments}
I would like to thank my supervisor Tatsuki Kuwagaki for many discussions and advice. I also would like to thank Daiki Irikura for having many discussions and answering my questions on topology. This work was supported by JST SPRING, Grant Number JPMJSP2110.

\section{Preliminaries}

\subsection{Symplectic Geometry}

Let $M$ be a differentiable manifold, $0_M$ be the zero section of the tangent bundle $TM$ and the cotangent bundle $T^*M$ of $M$, and $I$ be an open interval of $\mathbb{R}$ containing $0$ parametrized by time $t$. Set $\dot{T}M :  =TM \setminus 0_M$ and $\dot{T}^*M :  =T^*M\setminus 0_M$. The cotangent bundle $T^*M$ has the canonical exact symplectic structure given by $\alpha = \sum_{i} \xi_i dx_i$ and $\omega = \sum_{i} dx_i \wedge d\xi_i $. Let $(X,\omega)$ be a symplectic manifold. A Hamiltonian function $ H\colon X \to \mathbb{R} $ induces a Hamiltonian vector field $v \in \mathfrak{X}(X)$ by the equation $\omega(v,-)=dH$. A Hamiltonian vector field $v \in \mathfrak{X}(X) $ induces a Hamiltonian isotopy $\phi \colon  X \times I \to X $ by integrating $v$.

\begin{dfn}
The \emph{geodesic flow} on a complete Riemannian manifold $(M,g)$ is an isotopy $\phi \colon  \dot{T}M \times \mathbb{R} \to \dot{T}M $, $\phi_t(x,v) = d_{ \frac{v}{|v|} } {\Exp_{x,t}}( v )$ where $ \Exp_{x,t}\colon  T_xM \to M$ is the exponential map by time $t$ and $ d_{\frac{v}{|v|}}{\Exp_{x,t}} \colon  T_{\frac{v}{|v|}}(T_xM) \to T_{\Exp_{x,t}(\frac{v}{|v|})}M$ is its derivative.
\end{dfn}

\begin{example}
If we consider the Hamiltonian function $H\colon \dot{T}^*\mathbb{R}^n  \to \mathbb{R}$, $H(x,\xi)=|\xi|$, then the induced Hamiltonian isotopy is given by $\phi \colon  \dot{T}^*\mathbb{R}^n \times \mathbb{R} \to \dot{T}^*\mathbb{R}^n $, $\phi_t(x,\xi)=(x+t\frac{\xi}{|\xi|} , \xi)$. In general, if we consider a complete Riemannian manifold $(M,g)$ and the Hamiltonian function $H\colon  \dot{T}^*M \to \mathbb{R}$, $H(x,\xi)=|\xi|_{g^*}$, then the induced Hamiltonian isotopy is given by the geodesic flow on $(M,g)$ where $g^*$ is the Riemannian metric on $T^*M$ induced by $g$ and we identify $TM$ and $T^*M$ by $g$.
\end{example}

\begin{remark}
Consider a complete Riemannian manifold $(M,g)$.
The Hamiltonian function $H\colon \dot{T}^*M \to \mathbb{R}$, $H(x,\xi)= { |\xi|_{g^*} }$ induces the Hamiltonian isotopy $\phi \colon  \dot{T}^*M \times \mathbb{R} \to \dot{T}^*M$, $\phi_t(x,v) = d_{\frac{v}{|v|}} {\Exp_{x,t}}( v )$. On the other hand, the Hamiltonian function $H\colon \dot{T}^*M \to \mathbb{R}$, $H(x,\xi)= \frac{1}{2}|\xi|_{g^*}^2$ induces the Hamiltonian isotopy $\phi \colon  \dot{T}^*M \times \mathbb{R} \to \dot{T}^*M$, $\phi_t(x,v) = d_{v} {\Exp_{x,t}}( v )$. Note that we use the identification of $TM$ and $T^*M$ by $g$.
The former $\phi$ is referred to as the normalized geodesic flow and the latter $\phi$ is referred to as the geodesic flow. In this paper, we consider only the former and simply call it the geodesic flow. There are no essential differences between these flows because their restrictions to the unit sphere bundle coincide.
\end{remark}

\begin{dfn}
A Hamiltonian isotopy $\phi \colon  \dot{T}^*M \times I \to \dot{T}^*M $ is \emph{homogeneous} if $\phi_t \colon  \dot{T}^*M \to \dot{T}^*M$ commutes with the canonical $\mathbb{R}_{>0}$-action for all $t \in I$.
\end{dfn}

\subsection{Sheaf Theory}

Let $\mathbb{K}$ be a commutative unital ring and $\Mod(\mathbb{K})$ be the dg derived category of $\mathbb{K}$-modules. Let $X$ be a topological space and $\Sh(X)$ be the dg derived category of sheaves of $\mathbb{K}$-modules on $X$. We will use the term sheaf for an object of $\Sh(X)$.
There are six dg derived functors: $f_*$, $f^{-1}$, $f_!$, $f^!$, $ \otimes $, and $\mathcal{H}om$, where $f$ is a continuous map between locally compact Hausdorff spaces. For a sheaf $F\in \Sh(X)$ and a locally closed subset $ Z \subset X $, we set $ F_Z : = F \otimes \mathbb{K}_Z $ and $\Gamma_Z(F): = \mathcal{H}om (\mathbb{K}_Z, F) $. For sheaves $F,G \in \Sh(X)$, we set $\Ext^i(F,G): =H^i(\Hom_{\Sh(X)}(F,G))$. The notation basically follows \cite{KS}.

\begin{prop}[{\cite[p.115]{KS}}]
\label{prop2.5}
Let $X$ be a locally compact Hausdorff space, $Z_1$ be an open subset of $X$, and set $Z_2 :  = X \setminus Z_1$. Let $F$ be a sheaf on $X$. Then there exist distinguished triangles in $\Sh(X)$,
\begin{equation*}
F_{Z_1} \to F \to F_{Z_2} \xrightarrow{+1} 
\end{equation*}
and
\begin{equation*}
\Gamma_{Z_2}(F) \to F \to \Gamma_{Z_1}(F) \xrightarrow{+1}.
\end{equation*}
\end{prop}

\begin{prop}
\label{prop2.6}
Let $X$ be a manifold, $Z_1$ be an open subset of $X$, and $Z_2$ be a closed subset of $X$. Then the object $\Hom_{\Sh(X)}(\mathbb{K}_{Z_1}, \mathbb{K}_{Z_2}) $ is isomorphic to $ C^*(Z_1 \cap Z_2)$ in $\Mod(\mathbb{K})$ where $C^*(Z_1 \cap Z_2)$ is the singular cochain complex of $Z_1 \cap Z_2$ with coefficients in $\mathbb{K}$.
\end{prop}
\begin{proof}
Let $j\colon Z_1 \to X$ be the open embedding and $i\colon Z_1 \cap Z_2 \to Z_1$ be the closed embedding. 
Then we have the following chain of isomorphisms,
\begin{equation*}
\begin{split}
\Hom_{\Sh(X)}(\mathbb{K}_{Z_1}, \mathbb{K}_{Z_2})
&\cong \Hom_{\Sh(X)}(j_! \mathbb{K}_{Z_1}, \mathbb{K}_{Z_2}) \\
&\cong \Hom_{\Sh(Z_1)}(\mathbb{K}_{Z_1}, j^! \mathbb{K}_{Z_2}) \\
&\cong \Hom_{\Sh(Z_1)}(\mathbb{K}_{Z_1}, \mathbb{K}_{Z_1 \cap Z_2}) \\
&\cong \Hom_{\Sh(Z_1)}(\mathbb{K}_{Z_1}, i_*  \mathbb{K}_{Z_1 \cap Z_2}) \\
&\cong \Hom_{\Sh(Z_1 \cap Z_2)}( i^{-1}\mathbb{K}_{Z_1},  \mathbb{K}_{Z_1 \cap Z_2}) \\
&\cong \Hom_{\Sh(Z_1 \cap Z_2)}(\mathbb{K}_{Z_1 \cap Z_2}, \mathbb{K}_{Z_1 \cap Z_2}) \\
&\cong \Gamma(Z_1 \cap Z_2 ; \mathbb{K}_{Z_1 \cap Z_2})\\
&\cong C^*(Z_1 \cap Z_2).
\end{split}
\end{equation*}
The last isomorphism follows from the fact that the singular cochain complex computes sheaf cohomology.
\end{proof}

\begin{prop}
\label{prop2.7}
Let $X$ be a topological space and $A$ be a subset of $X$. Then there exists a distinguished triangle $C^*(X,A) \to C^*(X) \to C^*(A) \xrightarrow{+1}$ in $\Mod(\mathbb{K})$ where $C^*(X,A)$ is the relative singular cochain complex of the pair $(X,A)$.
\begin{proof}
This is a special case of the fact that a short exact sequence induces a distinguished triangle.
\end{proof}
\end{prop}

\begin{dfn}
Let $M$ be a differentiable manifold. For a sheaf $F \in \Sh(M) $, the \emph{microsupport} $\mathrm{SS}(F)$ is the closed subset of $T^*M$ defined by 
\begin{equation*}
\overline{\{ (x,\xi)\in T^*M \mid \exists f \in C^\infty (M) , f(x)=0, d_xf=\xi, (\Gamma_{\{x \in M \mid f(x) \geq 0\}}(F))_x \neq 0 \} } .
\end{equation*}
\end{dfn}

\begin{example}
Let $M$ be a differentiable manifold and $\phi\colon  M \to \mathbb{R}$ be a $C^\infty$-function such that $0 \in \mathbb{R}$ is a regular value of $\phi$. Set $U :  =\{x \in M \mid \phi(x) < 0 \}$. Then $ \mathrm{SS} (\mathbb{K}_U) = 0_U \cup \{ (x,\xi) \in T^*M \mid x \in \partial U, \xi=\lambda d_x\phi , \lambda \geq 0 \} $.
\end{example}

\begin{example}
Let $M$ be a differentiable manifold and $\phi\colon  M \to \mathbb{R}$ be a $C^\infty$-function such that $0 \in \mathbb{R}$ is a regular value of $\phi$. Set $Z :  =\{x \in M \mid \phi(x) \geq 0 \}$. Then $\mathrm{SS} (\mathbb{K}_Z) = 0_Z \cup \{ (x,\xi) \in T^*M \mid x \in \partial Z, \xi=\lambda d_x\phi , \lambda \geq 0 \} $.
\end{example}

\begin{prop}[{\cite[Proposition 5.1.3]{KS}}]
\label{prop2.11}
Let $M$ be a differentiable manifold. Let $F \to G \to H \xrightarrow{+1} $ be a distinguished triangle in $\Sh(M)$. Then $\mathrm{SS}(G) \subset \mathrm{SS}(F) \cup \mathrm{SS}(H)$.
\end{prop}

\begin{dfn}
For a homogeneous Hamiltonian isotopy $\phi \colon  \dot{T}^*M \times I \to \dot{T}^*M $, a \emph{GKS kernel} of $\phi$ is an object $K_\phi \in \Sh(M\times M \times I)$ satisfying the conditions,
\begin{equation*}
\mathrm{SS} (K_\phi) \setminus 0_{M \times M \times I} = \Lambda_\phi 
\end{equation*}
and
\begin{equation*}
K_\phi|_{ \{t=0\} } \cong \mathbb{K}_{\Delta_M}
\end{equation*}
where $\Lambda_\phi$ is the Lagrangian submanifold of $T^*(M \times M \times I)$ defined by
\begin{equation*}
\Lambda_\phi = \{ (\phi_t(x,\xi),(x,-\xi), (t, - \langle \alpha, v_\phi\rangle (\phi_t(x,\xi),t)  )  ) \in \dot{T}^*M \times \dot{T}^*M \times T^*I \mid (x,\xi)\in \dot{T}^*M , t \in I \},
\end{equation*}
$\alpha$ is the 1-form defined by $\alpha = \sum_i \xi_i dx_i$, and $v_\phi$ is the time-dependent vector field induced by $\phi$.
\end{dfn}

Note that if $H\colon  \dot{T}^*M \to \mathbb{R}$ generates a Hamiltonian isotopy $\phi\colon \dot{T}^*M \times I \to \dot{T}^*M$, then $H(x,\xi)= \langle \alpha, v_\phi \rangle(x,\xi,t)$. For example, the Hamiltonian function $H\colon \dot{T}^*\mathbb{R}^n \to \mathbb{R}$, $H(x,\xi)=|\xi|$ gives a Lagrangian submanifold $\Lambda_\phi =\{ ( (x+t\frac{\xi}{|\xi|},\xi),(x,-\xi), (t, -|\xi| )  )  \mid (x,\xi)\in \dot{T}^*\mathbb{R}^n , t \in \mathbb{R} \}$. In general, for a complete Riemannian manifold $(M,g)$, the Hamiltonian function $H\colon \dot{T}^*M \to \mathbb{R}$, $H(x,\xi)=|\xi|_{g^*}$ gives a Lagrangian submanifold $\Lambda_\phi =\{ ( \phi_t(x,\xi),(x,-\xi), (t, -|\xi|_{g^*} )  )  \mid (x,\xi)\in \dot{T}^*M , t \in \mathbb{R} \}$.

\begin{thm}[\cite{GKS}]
Let $\phi \colon  \dot{T}^*M \times I \to \dot{T}^*M $ be a homogeneous Hamiltonian isotopy.
Then a GKS kernel of $\phi$ exists and it is unique up to unique isomorphism in the homotopy category.
\end{thm}

\begin{example}
\label{example}
Consider the Hamiltonian function $H\colon \dot{T}^*\mathbb{R}^n \to \mathbb{R}$, $H(x,\xi)=|\xi| $ and the induced Hamiltonian isotopy $\phi\colon \dot{T}^*\mathbb{R}^n\times \mathbb{R} \to \dot{T}^*\mathbb{R}^n$, $\phi_t(x,\xi)=(x+t \frac{\xi}{|\xi|}, \xi)$.
Then the GKS kernel $K_\phi \in \Sh(\mathbb{R}^n \times \mathbb{R}^n \times  \mathbb{R} ) $ is given by 
\begin{equation*} 
\Cocone(\mathbb{K}_{\{(x,y,t) \mid |x-y|\leq -t\}} \xrightarrow{\psi} \mathbb{K}_{\{(x,y,t) \mid |x-y|<t \}}[n+1])
\end{equation*}
where $\psi \in \Hom^0_{\Sh(\mathbb{R}^n \times \mathbb{R}^n \times \mathbb{R})}(\mathbb{K}_{\{(x,y,t) \mid |x-y|\leq -t\}} ,\mathbb{K}_{\{(x,y,t) \mid |x-y|<t \}}[n
+1])$ is a morphism such that $H^0(\psi) \in  \Ext^{n+1} ( \mathbb{K}_{\{(x,y,t) \mid |x-y|\leq -t\}} ,\mathbb{K}_{\{(x,y,t) \mid |x-y|<t \}} )  \cong \mathbb{K}$ is a generator and we set $\Cocone(\cdot) :  =\Cone(\cdot)[-1]$.
Note that this example is essentially the same as Proposition \ref{prop1.3}.
\end{example}

\begin{dfn}
Let $f\colon X \to Y$ be a morphism in $\mathrm{Mod}(\mathbb{K})$ or $\mathrm{Sh}(M)$. We define $\Cocone(f) :  =\Cone(f)[-1]$.
\end{dfn}

\section{GKS kernel of the geodesic flow on $S^n$}

We fix the standard metric $g$ on $S^n$. Note that the diameter of $S^n$ is $\pi$ and the period of the geodesic flow on $S^n$ is $2\pi$. We consider the Hamiltonian function $H\colon \dot{T}^*S^n \to \mathbb{R} , H(x,\xi)=|\xi|_{g^*}$ and the induced Hamiltonian isotopy $\phi \colon  \dot{T}^*S^n \times \mathbb{R} \to \dot{T}^*S^n$. In this section, we shall explicitly construct the GKS kernel $K_\phi$ of the geodesic flow $\phi$. Note that the cohomology of $K_\phi$ is already known in \cite{GKS}. \\

For $i \in \mathbb{Z}_{>0}$, we define the open subset $Z_i$ of $S^n \times S^n \times \mathbb{R} $ by
\begin{equation*}
Z_i :  =\{(x,y,t)\in S^n \times S^n \times \mathbb{R} \mid \mathrm{dist}(x,(-1)^{i-1}y)<t-(i-1)\pi\}.
\end{equation*}

\begin{lem}
\label{lem3.1}
For $i \in\mathbb{Z}_{>0}$ and $\epsilon > 0$, the inclusion map $ j \colon  S^n \times S^n \times \{i\pi +\epsilon \} \to Z_i$ is a homotopy equivalence.
\end{lem}
\begin{proof}
The natural map $r\colon  Z_i \to S^n \times S^n \times \{ i\pi + \epsilon \}$, $r (x,y,t) = (x, y, i\pi + \epsilon)$ is a homotopy equivalence because a homotopy $H\colon Z_i \times [0,1] \to Z_i$ between $\mathrm{id}_{Z_i}$ and $j \circ r$ is given by $H(x,y,t,s)=(x,y,(1-s)t + s(i\pi + \epsilon) )$.
\end{proof}

\begin{lem}
\label{lem3.2}
For $i \in \mathbb{Z}_{>0}$ odd and $0<\epsilon<\pi$, the inclusion map $j \colon  \Delta_{S^n} \times \{i\pi +\epsilon \} \to Z_i \setminus Z_{i+1}$ is a homotopy equivalence where $\Delta_{S^n}$ is the diagonal set of $S^n \times S^n$. For $ i \in \mathbb{Z}_{>0}$ even and $0<\epsilon<\pi$, the inclusion map $j \colon  -\Delta_{S^n} \times \{i\pi +\epsilon\} \to Z_i \setminus Z_{i+1}$ is a homotopy equivalence where $ -\Delta_{S^n}$ is the anti-diagonal set of $S^n \times S^n$
\end{lem}
\begin{proof}
Consider the case of $i$ odd. For $(x,y) \notin \Delta_{S^n} \cup -\Delta_{S^n}$, we define $v_x\in T_y{S^n}$ by the equations $\Exp_{y,\mathrm{dist}(x,y)}(v_x) = x$ and $|v_x|_g=1$. For $(x,y) \in \Delta_{S^n}$, we define $v_x \in T_y{S^n}$ by $v_x=0$. Note that $v_x$ is determined uniquely. The natural map $r\colon  Z_i \setminus Z_{i+1} \to \Delta_{S^n} \times \{ i\pi + \epsilon \}$, $r (x,y,t) = (x, x, i\pi + \epsilon)$ is a homotopy equivalence because a homotopy $H\colon Z_i\setminus Z_{i+1} \times [0,1] \to Z_i \setminus Z_{i+1}$ between $\mathrm{id}_{Z_i \setminus Z_{i+1}}$ and $j \circ r$ is given by $H(x,y,t,s)=(x, \Exp_{ y,s\mathrm{dist}(x,y) }(v_x) ,(1-s)t + s(i\pi + \epsilon) )$.

Consider the case of $i$ even. For $(x,y) \notin \Delta_{S^n} \cup -\Delta_{S^n}$, we define $v_x\in T_y{S^n}$ by the equations $\Exp_{y,\mathrm{dist}(x,y)}(v_x) = x$ and $|v_x|_g=1$. For $(x,y) \in -\Delta_{S^n}$, we define $v_x \in T_y{S^n}$ by $v_x=0$. Note that $v_x$ is determined uniquely. The natural map $r\colon  Z_i \setminus Z_{i+1} \to -\Delta_{S^n} \times \{ i\pi + \epsilon \}$, $r (x,y,t) = (x, -x, i\pi + \epsilon)$ is a homotopy equivalence because a homotopy $H\colon Z_i\setminus Z_{i+1} \times [0,1] \to Z_i \setminus Z_{i+1}$ between $\mathrm{id}_{Z_i \setminus Z_{i+1}}$ and $j \circ r$ is given by $H(x,y,t,s)=(x, \Exp_{ y,s (\pi-\mathrm{dist}(x,y) )}(-v_x) ,(1-s)t + s(i\pi + \epsilon) )$.
\end{proof}

\begin{lem}
\label{lem3.3}
For $i \in \mathbb{Z}_{>0}$ odd, the object $\Hom_{\Sh(S^n \times S^n \times \mathbb{R})}(\mathbb{K}_{Z_i} , \mathbb{K}_{Z_{i+1}}) $ is isomorphic to $ C^*(S^n \times S^n , S^n \times S^n \setminus -\Delta_{S^n})$ in $\Mod(\mathbb{K})$. For $i \in \mathbb{Z}_{>0}$ even, the object $\Hom_{\Sh(S^n \times S^n \times \mathbb{R})}(\mathbb{K}_{Z_i} , \mathbb{K}_{Z_{i+1}}) $ is isomorphic to $ C^*(S^n \times S^n, S^n \times S^n \setminus \Delta_{S^n} )$ in $\Mod(\mathbb{K})$.
\end{lem}
\begin{proof}
We only prove the case of $i$ odd since the proof for $i$ even is similar. We have the following chain of isomorphisms,
\begin{equation*}
\begin{split}
&\Hom_{\Sh(S^n \times S^n \times \mathbb{R})}(\mathbb{K}_{Z_i},\mathbb{K}_{Z_{i+1}}) \\
&\cong \Cocone(\Hom_{\Sh(S^n \times S^n \times \mathbb{R})}(\mathbb{K}_{Z_i}, \mathbb{K}_{S^n \times S^n \times \mathbb{R} }) \to \Hom_{\Sh(S^n \times S^n \times \mathbb{R})}(\mathbb{K}_{Z_i} , \mathbb{K}_{{Z_{i+1}}^c})) \\
&\cong \Cocone(C^*(Z_i) \to C^*(Z_i \setminus { Z_{i+1} })) \\
&\cong \Cocone(C^*(S^n \times S^n) \to C^*(\Delta_{S^n})) \\
&\cong C^*(S^n \times S^n , \Delta_{S^n}) \\
&\cong C^*(S^n \times S^n , S^n \times S^n \setminus -\Delta_{S^n}).
\end{split}
\end{equation*}
The first isomorphism follows from Proposition \ref{prop2.5}.
The second follows from Proposition \ref{prop2.6}.
The third follows from Lemma \ref{lem3.1} and \ref{lem3.2}.
The fourth follows from Proposition \ref{prop2.7}. 
The fifth follows from the fact that the inclusion map $\Delta_{S^n} \to S^n \times S^n \setminus -\Delta_{S^n}$ is a homotopy equivalence.

\end{proof}

For $i \in \mathbb{Z}_{>0}$, we choose a morphism $\psi_i \in \Hom^0_{\Sh(S^n \times S^n \times \mathbb{R})} ( \mathbb{K}_{Z_i} , \mathbb{K}_{Z_{i+1}}[n] ) $ such that $H^0(\psi_i) \in \Ext^n(\mathbb{K}_{Z_i}, \mathbb{K}_{Z_{i+1}}) \cong \mathbb{K}$ is a generator. Note that the isomorphism $\Ext^n(\mathbb{K}_{Z_i}, \mathbb{K}_{Z_{i+1}}) \cong \mathbb{K}$ follows from Lemma \ref{lem3.3}. Thus we obtain the sequence in $ \Sh(S^n \times S^n \times \mathbb{R})$,
\begin{equation*}
\mathbb{K}_{Z_1} \xrightarrow{\psi_1} \mathbb{K}_{Z_2}[n] \xrightarrow{\psi_2} \mathbb{K}_{Z_3}[2n] \xrightarrow{\psi_3} \mathbb{K}_{Z_4}[3n] \xrightarrow{\psi_4} \cdots .
\end{equation*}

\begin{lem}
\label{lem3.4}
Set $A :  ={\{(x,t) \in \mathbb{R}^n \times \mathbb{R} \mid |x| > -t\}}$ and $B :  =\{(x,t) \in \mathbb{R}^n \times \mathbb{R} \mid |x| < t\}$. Let $ \theta \in \Hom^0_{\Sh(\mathbb{R}^n \times \mathbb{R})} ( \mathbb{K}_A , \mathbb{K}_B[n] ) $ be a morphism such that $H^0(\theta) \in \Ext^n(\mathbb{K}_A, \mathbb{K}_B) \cong \mathbb{K}$ is a generator. Then 
\begin{equation*}
\begin{split}
\mathrm{SS} ( \mathbb{K}_A ) \cap T^*_{(0,0)}(\mathbb{R}^n \times \mathbb{R}) = \{ a_1 dx_1 + \dots + a_{n+1} dt \mid a_{n+1} \leq 0 , {a_1}^2 + \dots + {a_n}^2 \leq {a_{n+1}}^2 \} , \\
\mathrm{ SS } (\mathbb{K}_B ) \cap T^*_{(0,0)}(\mathbb{R}^n \times \mathbb{R}) = \{ a_1 dx_1 + \dots + a_{n+1} dt \mid a_{n+1} \leq 0 , {a_1}^2 + \dots + {a_n}^2 \leq {a_{n+1}}^2 \} , \\
\mathrm{SS} ( \Cocone(\theta) ) \cap T^*_{(0,0)}(\mathbb{R}^n \times \mathbb{R}) = \{ a_1 dx_1 + \dots + a_{n+1} dt \mid a_{n+1} \leq 0 , {a_1}^2 + \dots + {a_n}^2 = {a_{n+1}}^2 \}.
\end{split}
\end{equation*}
\end{lem}
\begin{proof}
For a convex cone $\gamma \subset \mathbb{R}^n \times \mathbb{R}$, we denote its dual cone by $\gamma^\vee$. 
We denote the open ball of radius $\epsilon$ and center $x$ by $B(x;\epsilon)$.
We identify $ T^*_{(0,0)} (\mathbb{R}^n \times \mathbb{R})$ and the dual vector space of $\mathbb{R}^n \times \mathbb{R}$.
Let $\mathbb{D}\colon \mathrm{Sh}(\mathbb{R}^n \times \mathbb{R})\to\mathrm{Sh}(\mathbb{R}^n \times \mathbb{R}) $ be the Verdier dual functor.

\begin{enumerate}

\item[(i)]

Note that $A^c$ is a closed convex cone in $\mathbb{R}^n \times \mathbb{R}$. 
Then we have
\begin{align*}
&\mathrm{SS} ( \mathbb{K}_A ) \cap T^*_{(0,0)}(\mathbb{R}^n \times \mathbb{R})  \\
&= \mathrm{SS} ( \mathbb{K}_{A^c} ) \cap T^*_{(0,0)}(\mathbb{R}^n \times \mathbb{R})\\
&= (A^c)^\vee \\
&= \{(x,t) \in \mathbb{R}^n \times \mathbb{R} \mid |x|\leq -t \}^\vee \\
&= \{ a_1 dx_1 + \dots + a_{n+1} dt \in T^*_{(0,0)}(\mathbb{R}^n \times \mathbb{R}) \mid a_{n+1} \leq 0 , {a_1}^2 + \dots + {a_n}^2 \leq {a_{n+1}}^2 \}. 
\end{align*}
The second equality follows from \cite[Proposition 5.3.1]{KS}.

\item[(ii)]

Note that $B$ is a convex open cone in $\mathbb{R}^n \times \mathbb{R}$. Then we have
\begin{align*}
&\mathrm{SS}(\mathbb{K}_B) \cap T^*_{(0,0)}(\mathbb{R}^{n} \times \mathbb{R}) \\
&= \mathrm{SS}(\mathbb{D}({\mathbb{K}_{\overline{B}}})) \cap T^*_{(0,0)}(\mathbb{R}^n \times \mathbb{R}) \\
&= -(\mathrm{SS}(\mathbb{K}_{\overline{B}}) ) \cap T^*_{(0,0)}(\mathbb{R}^n \times \mathbb{R}) \\
&= - \overline{B}^\vee \\
&= - \{(x,t) \in \mathbb{R}^n \times \mathbb{R} \mid |x| \leq t \}^\vee\\
&= \{(x,t) \in \mathbb{R}^n \times \mathbb{R} \mid |x| \leq -t \}^\vee\\
&= \{a_1dx_1+\dots + a_{n+1}dx_{n+1} \in T^*_{(0,0)}(\mathbb{R}^n \times \mathbb{R}) \mid a_{n+1} \leq 0, {a_1}^2+\dots + {a_n}^2  \leq {a_{n+1}}^2 \}.
\end{align*}
The second equality follows from the fact that Verdier duality transforms the microsupport by the antipodal map.

\item[(iii)]
Consider $\xi= a_1d x_1+\dots+a_{n+1}dt  \in T^*_{(0,0)}(\mathbb{R}^n \times \mathbb{R} )$ and a $C^\infty$-function $f\colon \mathbb{R}^n \times \mathbb{R} \to \mathbb{R}$ such that  $f(0,0)=0$ and $d_{(0,0)}f=\xi$. 
Let $i:\{f\geq 0\} \to \mathbb{R}^n \times \mathbb{R}$ be the inclusion map.
If $a_{n+1} < 0 $ and $ {a_1}^2 + \dots + {a_n}^2 < {a_{n+1}}^2$, then there exists a sufficiently small $\epsilon>0$ such that $B \cap B(0;\epsilon) \subset \{f<0\} \cap B(0;\epsilon) \subset A \cap B(0;\epsilon)$ and we have the following commutative diagram: 

\begin{equation*}
\begin{tikzcd}
(\Gamma_{\{f \geq 0 \}}\mathbb{K}_A)_{(0,0)}  \ar[r] \ar[d,"\cong"] & (\Gamma_{\{f \geq 0 \}}\mathbb{K}_B[n])_{(0,0)}  \ar[d,"\cong"] \\
\Gamma(B(0;\epsilon) ; i_*i^{!}\mathbb{K}_A )  \ar[d,"\cong"] & \Gamma(B(0;\epsilon); i_*i^{!}\mathbb{K}_B[n]) \ar[d,"\cong"] \\ 
\Hom_{\Sh(\mathbb{R}^n \times \mathbb{R})}(\mathbb{K}_{B(0;\epsilon)}, i_*i^{!}\mathbb{K}_A)  \ar[d,"\cong"] & \Hom_{\Sh(\mathbb{R}^n \times \mathbb{R})}(\mathbb{K}_{B(0;\epsilon) }, i_*i^{!} \mathbb{K}_B[n]) \ar[d,"\cong"] \\
\Hom_{\Sh(\mathbb{R}^n \times \mathbb{R})}(i_!i^{-1}\mathbb{K}_{B(0;\epsilon)},\mathbb{K}_A)  \ar[d,"\cong"] & \Hom_{\Sh(\mathbb{R}^n \times \mathbb{R})}(i_! i^{-1} \mathbb{K}_{B(0;\epsilon)}, \mathbb{K}_B[n]) \ar[d,"\cong"] \\
\Hom_{\Sh(\mathbb{R}^n \times \mathbb{R})}(\mathbb{K}_{ \{ f \geq 0\} \cap B(0;\epsilon) } , \mathbb{K}_{A} )  \ar[d,"\cong"] & \Hom_{ \Sh(\mathbb{R}^n \times \mathbb{R})}(\mathbb{K}_{\{f \geq 0\} \cap B(0;\epsilon)} , \mathbb{K}_{B}[n] ) \ar[d,"\cong"] \\
\Hom_{\Sh(\mathbb{R}^n \times \mathbb{R})}(\mathbb{K}_{ \{ f < 0 \} \cap B(0;\epsilon) } , \mathbb{K}_{A})  \ar[d,"\cong"][-1] & \Hom_{\Sh(\mathbb{R}^n \times \mathbb{R})}(\mathbb{K}_{\{f < 0 \} \cap B(0;\epsilon)} , \mathbb{K}_{B}[n])[-1] \ar[d,"\cong"] \\
\Hom_{\Sh(\mathbb{R}^n \times \mathbb{R})}(\mathbb{K}_{A}, \mathbb{K}_{A})[-1] \ar[r, "\theta \circ -"] & \Hom_{\Sh(\mathbb{R}^n \times \mathbb{R})}(\mathbb{K}_{A}, \mathbb{K}_{B}[n]) [-1].
\end{tikzcd}
\end{equation*}
By the definition of $\theta$, the morphism $(\Gamma_{\{f \geq 0 \}} \mathbb{K}_A)_{(0,0)} \to (\Gamma_{\{ f \geq 0 \}}\mathbb{K}_B[n])_{(0,0)}$ is an isomorphism.
Therefore we have $\xi \notin \mathrm{SS}(\Cocone(\theta))$ and
\begin{equation*}
\mathrm{SS} ( \Cocone(\theta) ) \cap T^*_{(0,0)}(\mathbb{R}^n \times \mathbb{R}) \subset \{ a_1 dx_1 + \dots + a_{n+1} dt \mid a_{n+1} \leq 0 , {a_1}^2 + \dots + {a_n}^2 = {a_{n+1}}^2 \}.
\end{equation*}
The converse
\begin{equation*}
\mathrm{SS} ( \Cocone(\theta) ) \cap T^*_{(0,0)}(\mathbb{R}^n \times \mathbb{R}) \supset \{ a_1 dx_1 + \dots + a_{n+1} dt \mid a_{n+1} \leq 0 , {a_1}^2 + \dots + {a_n}^2 = {a_{n+1}}^2 \}
\end{equation*}
follows by a direct computation.
\end{enumerate}

\end{proof}

\begin{prop}
\label{prop3.5}
For $i \in \mathbb{Z}_{>0}$ odd, there is an equality,
\begin{equation*}
\begin{split}
\mathrm{SS} (\mathbb{K}_{Z_i}) \cap \dot{T}^*_{(x,x,(i-1)\pi)}(S^n \times S^n \times \mathbb{R})&= \left\{ \sum_k \xi_k \middle| \xi_k \in \Lambda_\phi \cap \dot{T}^*_{(x,x,(i-1)\pi)}(S^n \times S^n \times \mathbb{R})  \right\}, \\
\mathrm{SS} (\mathbb{K}_{Z_i}) \cap \dot{T}^*_{(x,-x,i\pi)}(S^n \times S^n \times \mathbb{R}) &= \left\{ \sum_k \xi_k \middle| \xi_k \in \Lambda_\phi \cap \dot{T}^*_{(x,-x,i\pi)}(S^n \times S^n \times \mathbb{R}) \right\}, \\
\mathrm{SS} (\mathbb{K}_{Z_i}) \cap \dot{T}^*_{(x,y,t)}(S^n \times S^n \times \mathbb{R}) &= \Lambda_{\phi} \cap \dot{T}^*_{(x,y,t)}(S^n \times S^n \times \mathbb{R} ) \ \text{(if $(i-1)\pi < t < i\pi$)}, \\
\mathrm{SS} (\mathbb{K}_{Z_i}) \cap \dot{T}^*_{(x,y,t)}(S^n \times S^n \times \mathbb{R}) &= \emptyset \ \text{(otherwise)}.
\end{split}
\end{equation*}
For $i \in \mathbb{Z}_{>0}$ even, there is an equality,
\begin{equation*}
\begin{split}
\mathrm{SS} (\mathbb{K}_{Z_i}) \cap \dot{T}^*_{(x,-x,(i-1)\pi)}(S^n \times S^n \times \mathbb{R})&= \left\{ \sum_k \xi_k \middle| \xi_k \in \Lambda_\phi \cap \dot{T}^*_{(x,-x,(i-1)\pi)}(S^n \times S^n \times \mathbb{R})  \right\}, \\
\mathrm{SS} (\mathbb{K}_{Z_i}) \cap \dot{T}^*_{(x,x,i\pi)}(S^n \times S^n \times \mathbb{R}) &= \left\{ \sum_k \xi_k \middle| \xi_k \in \Lambda_\phi \cap \dot{T}^*_{(x,x,i\pi)}(S^n \times S^n \times \mathbb{R}) \right\}, \\
\mathrm{SS} (\mathbb{K}_{Z_i}) \cap \dot{T}^*_{(x,y,t)}(S^n \times S^n \times \mathbb{R}) &= \Lambda_{\phi} \cap \dot{T}^*_{(x,y,t)}(S^n \times S^n \times \mathbb{R} )  \ \text{(if $(i-1)\pi<t<i\pi$)}, \\
\mathrm{SS} (\mathbb{K}_{Z_i}) \cap \dot{T}^*_{(x,y,t)}(S^n \times S^n \times \mathbb{R}) &= \emptyset \ \text{(otherwise)}.
\end{split}
\end{equation*}
\end{prop}
\begin{proof}
It follows from the definition of $Z_i$ and Lemma \ref{lem3.4}.
\end{proof}

\begin{prop}
\label{prop3.6}
For $i \in \mathbb{Z}_{>0}$ odd, there is an equality,
\begin{equation*}
\mathrm{SS} (\Cocone(\psi_i)) \cap \dot{T}^*_{(x,-x,i\pi)} (S^n \times S^n \times \mathbb{R} ) = \Lambda_\phi \cap \dot{T}^*_{(x,-x,i\pi)}(S^n \times S^n \times \mathbb{R}).
\end{equation*}
For $i \in \mathbb{Z}_{>0}$ even, there is an equality,
\begin{equation*}
\mathrm{SS} (\Cocone(\psi_i)) \cap \dot{T}^*_{(x,x,i\pi)} (S^n \times S^n \times \mathbb{R} ) = \Lambda_\phi \cap \dot{T}^*_{(x,x,i\pi)}(S^n \times S^n \times \mathbb{R}).
\end{equation*}
\end{prop}
\begin{proof}
We only prove the case of $i=1$ since the proof for $i > 1$ is similar.
Fix a point $x_0 \in S^n$.
Set $W : = \{x_0\} \times S^n \times (\pi-\epsilon, \pi+\epsilon)$ and $W' : = \{x_0\} \times B(-x_0;\epsilon) \times (\pi-\epsilon, \pi+\epsilon)$.
Since the microsupport is a local invariant, it suffices to work on a sufficiently small neighborhood of each point. 
Restricting to a neighborhood $B(x_0;\epsilon)$ of $x_0$, the subset $\bigcup_{x\in B(x_0;\epsilon)} \{x\}\times\{(y,t)\in S^n\times\mathbb{R}\mid \mathrm{dist}(x,y)<t\}$ is diffeomorphic to $\bigcup_{x\in B(x_0;\epsilon)}\{x\}\times\{(y,t)\in S^n\times\mathbb{R}\mid \mathrm{dist}(x_0,y)<t\}$.
Hence, the problem is locally reduced to the case where the first component is fixed.
Then it suffices to prove the equality $\mathrm{SS} (\Cocone(\psi_1 |_{W'})) \cap \dot{T}^*_{(x_0,-x_0,\pi)}W' = \{((-x_0,\xi) , (\pi, -|\xi|_{g^*})) \in \dot{T}^*_{-x_0}S^n \times T^*_{\pi}\mathbb{R}  \mid (-x_0, \xi) \in \dot{T}^*_{-x_0}S^n  \}$.
We have the following commutative diagram:
\begin{equation*}
\begin{tikzcd}[column sep=tiny]
\Hom(\mathbb{K}_{Z_1},\mathbb{K}_{Z_2}) \ar[r] \ar[d] & \Hom(\mathbb{K}_{Z_1}, \mathbb{K}_{S^n \times S^n \times \mathbb{R}}) \ar[r] \ar[d] & \Hom (\mathbb{K}_{Z_1}, \mathbb{K}_{{Z_2}^c}) \ar[d] \ar[r, "+1"] & \   \\
\Hom(\mathbb{K}_{Z_1}|_W,\mathbb{K}_{Z_2}|_W) \ar[r] \ar[d] & \Hom(\mathbb{K}_{Z_1}|_W, \mathbb{K}_{S^n \times S^n \times \mathbb{R}}|_W) \ar[r] \ar[d] & \Hom(\mathbb{K}_{Z_1}|_W, \mathbb{K}_{{Z_2}^c}|_W) \ar[d] \ar[r, "+1"] & \   \\
\Hom(\mathbb{K}_{Z_1}|_{W'},\mathbb{K}_{Z_2}|_{W'}) \ar[r] & \Hom(\mathbb{K}_{Z_1}|_{W'}, \mathbb{K}_{S^n \times S^n \times \mathbb{R}}|_{W'}) \ar[r] & \Hom(\mathbb{K}_{Z_1}|_{W'}, \mathbb{K}_{{Z_2}^c}|_{W'}) \ar[r, "+1"] & \ .   
\end{tikzcd}
\end{equation*}
The vertical arrows are the restriction maps. By Proposition \ref{prop2.7}, this diagram can be seen as the following commutative diagram:
\begin{equation*}
\begin{tikzcd}[column sep=tiny]
C^*( S^n \times S^n , S^n \times S^n \setminus -\Delta_{S^n}  ) \ar[r] \ar[d,"{i_1}^*"] & C^*(S^n \times S^n) \ar[r] \ar[d] & C^*( S^n \times S^n \setminus -\Delta_{S^n} ) \ar[d] \ar[r, "+1"] & \   \\
C^* (  S^n,  S^n \setminus \{-x_0\} ) \ar[r] \ar[d,"{i_2}^*"] & C^*( S^n ) \ar[r] \ar[d] & C^*( S^n \setminus \{-x_0\}  ) \ar[d] \ar[r, "+1"] & \   \\
C^*( B(-x_0; \epsilon), B(-x_0; \epsilon) \setminus \{-x_0\} ) \ar[r] & C^*(  B(-x_0; \epsilon )  ) \ar[r] & C^*( B(-x_0; \epsilon) \setminus \{-x_0\} ) \ar[r, "+1"] & \ 
\end{tikzcd}
\end{equation*}
where $i_1 \colon  S^n \to S^n \times S^n$ is the continuous map defined by $i_1(y)=(x_0,y)$ and $i_2 \colon  \overline{B(-x_0;\epsilon)} \to S^n$ is the inclusion map.
Taking its cohomology, we have the following commutative diagram:
\begin{equation*}
\begin{tikzcd}[column sep=tiny]
{H}^n(S^n \times S^n , S^n \times S^n \setminus -\Delta_{S^n} ) \ar[r] \ar[d,"H^n({i_1}^*)"] & H^n(S^n \times S^n) \ar[r] \ar[d] & H^n( S^n \times S^n \setminus -\Delta_{S^n}  ) \ar[d]  \\
{H}^n( S^n,  S^n \setminus \{-x_0\} ) \ar[r] \ar[d,"H^n({i_2}^*)"] & H^n(S^n) \ar[r] \ar[d] & H^n( S^n \setminus \{-x_0\}  )  \ar[d]  \\
{H}^n( B(-x_0; \epsilon ) , B(-x_0; \epsilon) \setminus \{-x_0\}  ) \ar[r] & H^n( B(-x_0;\epsilon)  ) \ar[r] & H^n ( B(-x_0; \epsilon) \setminus \{-x_0\} )  .
\end{tikzcd}
\end{equation*}
By diagram chasing, it follows that $H^n({i_1}^*)$ is an isomorphism. By the excision theorem, it follows that $H^n({i_2}^*)$ is an isomorphism. 
Therefore $\psi_1 |_{W'} \in \Hom^0 _{\Sh(W')}  ( \mathbb{K}_{Z_1}|_{W'} , \mathbb{K}_{Z_2}[n] |_{W'} )$ is a morphism such that $ H^0(\psi_1 |_{W'}) \in \Ext^n(\mathbb{K}_{Z_1} |_{W'}, \mathbb{K}_{Z_2} |_{W'}) \cong \mathbb{K}$ is a generator. Since the local situation is the same as Lemma \ref{lem3.4}, it follows that $\mathrm{SS} (\Cocone(\psi_1 |_{W'})) \cap \dot{T}^*_{(x_0,-x_0,\pi)}W' = \{((-x_0,\xi) , (\pi, -|\xi|_{g^*})) \in \dot{T}^*_{-x_0}S^n \times T^*_{\pi}\mathbb{R}  \mid (-x_0, \xi) \in \dot{T}^*_{-x_0}S^n  \}$.
\end{proof}

We define the sheaf $K_{+} \in \Sh(S^n \times S^n \times \mathbb{R})$ as follows.

For $i \in \mathbb{Z}_{>0}$, we define the sheaf $K_i \in \Sh(S^n \times S^n \times \mathbb{R})$ by
\begin{equation*}
\begin{split}
K_1& :  =\Cocone(\mathbb{K}_{Z_1} \xrightarrow{\psi_1} \mathbb{K}_{Z_2}[n]), \\
K_2& :  =\Cocone(K_1 \xrightarrow{\tilde{\psi}_2} \mathbb{K}_{Z_3}[2n-1]), \\
K_3& :  =\Cocone(K_2 \xrightarrow{\tilde{\psi}_3} \mathbb{K}_{Z_4}[3n-2]), \\
K_4& :  =\Cocone(K_3 \xrightarrow{\tilde{\psi}_4} \mathbb{K}_{Z_5}[4n-3]), \\
&\dots \\
K_i& :  =\Cocone(K_{i-1} \xrightarrow{\tilde{\psi}_i} \mathbb{K}_{Z_{i+1}}[in-(i-1)]) ,
\end{split}
\end{equation*}
where $\tilde{\psi}_i$ is the morphism induced by $\psi_i$.
Note that we have the following isomorphisms by induction.
\begin{enumerate}
\item [(i)] the case of $i=2$
\begin{equation*}
\begin{split}
&\Hom_{\Sh(S^n \times S^n \times \mathbb{R})}(K_1 , \mathbb{K}_{Z_{3} }[2n-1]) \\
&\cong \Cone(\Hom_{\Sh(S^n \times S^n \times \mathbb{R})}(\mathbb{K}_{Z_{2}}[n], \mathbb{K}_{Z_{3}}[2n-1]) \to \Hom_{\Sh(S^n \times S^n \times \mathbb{R})}(\mathbb{K}_{Z_1}, \mathbb{K}_{Z_{3}}[2n-1])) \\
&\cong \Cone(\Hom_{\Sh(S^n \times S^n \times \mathbb{R})}(\mathbb{K}_{Z_{2}}[n], \mathbb{K}_{Z_{3}}[2n-1])\to 0) \\
&\cong \Hom_{\Sh(S^n \times S^n \times \mathbb{R})}(\mathbb{K}_{Z_{2}}[n], \mathbb{K}_{Z_{3}}[2n-1])[1] \\
&\cong \Hom_{\Sh(S^n \times S^n \times \mathbb{R})}(\mathbb{K}_{Z_{2}}, \mathbb{K}_{Z_{3}}[n]) .
\end{split}
\end{equation*}
\item [(ii)] the case of $i>2$
\begin{equation*}
\begin{split}
&\Hom_{\Sh(S^n \times S^n \times \mathbb{R})}(K_{i-1} , \mathbb{K}_{Z_{i+1} }[in-(i-1)]) \\
&\cong \Cone(\Hom_{\Sh(S^n \times S^n \times \mathbb{R})}(\mathbb{K}_{Z_{i}}[(i-1)n-(i-2)], \mathbb{K}_{Z_{i+1}}[in-(i-1)]) \\
&\qquad \to \Hom_{\Sh(S^n \times S^n \times \mathbb{R})}(K_{i-2}, \mathbb{K}_{Z_{i+1}}[in-(i-1)])) \\
&\cong \Cone(\Hom_{\Sh(S^n \times S^n \times \mathbb{R})}(\mathbb{K}_{Z_{i}}[(i-1)n-(i-2)], \mathbb{K}_{Z_{i+1}}[in-(i-1)])\to 0) \\
&\cong \Hom_{\Sh(S^n \times S^n \times \mathbb{R})}(\mathbb{K}_{Z_{i}}[(i-1)n-(i-2)], \mathbb{K}_{Z_{i+1}}[in-(i-1)])[1] \\
&\cong \Hom_{\Sh(S^n \times S^n \times \mathbb{R})}(\mathbb{K}_{Z_{i}}, \mathbb{K}_{Z_{i+1}}[n]) .
\end{split}
\end{equation*}
\end{enumerate}
Finally we obtain the sequence in $ \Sh(S^n \times S^n \times \mathbb{R})$,
\begin{equation*}
K_1 \leftarrow K_2 \leftarrow K_3 \leftarrow K_4 \leftarrow \cdots .
\end{equation*}
We define $K_{+} :  = \mathrm{holim}_{i>0} {K_i}$. Note that $\mathrm{SS}(K_+)$ and $\Lambda_\phi$ coincide in $\dot{T}^*(S^n \times S^n \times \mathbb{R}_{>0})$.

In parallel to the construction of $K_+$, we define the sheaf $K_{-} \in \Sh(S^n \times S^n \times \mathbb{R}) $ as follows.

\begin{enumerate}

\item
For $i \in \mathbb{Z}_{<0}$, we define the closed subset $Z_i \subset S^n \times S^n \times \mathbb{R}$ by
\begin{equation*}
Z_{i} :  =\{(x,y,t)\in S^n \times S^n \times \mathbb{R} \mid \mathrm{dist}(x,(-1)^{i-1}y)\leq -t+(i+1)\pi \}.
\end{equation*}

\item
For $i \in \mathbb{Z}_{<0}$, we choose a morphism $\psi_i \in \Hom^0_{\Sh(S^n \times S^n \times \mathbb{R})} ( \mathbb{K}_{Z_{i-1}} , \mathbb{K}_{Z_i}[n] ) $ such that $H^0(\psi_i) \in \Ext^n(\mathbb{K}_{Z_{i-1}}, \mathbb{K}_{Z_i}) \cong \mathbb{K}$ is a generator.

\item
We obtain the sequence in $ \Sh(S^n \times S^n \times \mathbb{R})$,
\begin{equation*}
\cdots \xrightarrow{\psi_{-4}} \mathbb{K}_{Z_{-4}}[-3n] \xrightarrow{\psi_{-3}} \mathbb{K}_{Z_{-3}}[-2n] \xrightarrow{\psi_{-2}} \mathbb{K}_{Z_{-2}}[-n] \xrightarrow{\psi_{-1}} \mathbb{K}_{Z_{-1}} .
\end{equation*}

\item
For $i \in \mathbb{Z}_{<0}$, we define the sheaf $K_i \in \Sh(S^n \times S^n \times \mathbb{R})$ by
\begin{equation*}
\begin{split}
K_{-1}& :  =\Cone(\mathbb{K}_{Z_{-2}}[-n] \xrightarrow{\psi_{-1}} \mathbb{K}_{Z_{-1}}),\\
K_{-2}& :  =\Cone(\mathbb{K}_{Z_{-3}}[-2n+1] \xrightarrow{\tilde{\psi}_{-2}} K_{-1} ), \\
K_{-3}& :  =\Cone( \mathbb{K}_{Z_{-4}}[-3n+2] \xrightarrow{\tilde{\psi}_{-3}} K_{-2} ), \\
K_{-4}& :  =\Cone(\mathbb{K}_{Z_{-5}}[-4n+3] \xrightarrow{\tilde{\psi}_{-4}} K_{-3} ), \\
&\dots \\
K_i& :  =\Cone(\mathbb{K}_{Z_{i-1}}[in+(-i-1)] \xrightarrow{\tilde{\psi}_i} K_{i+1}),
\end{split}
\end{equation*}
where $ \tilde{\psi}_i $ is the morphism induced by $\psi_i$.

\item
We obtain the sequence in $\Sh(S^n \times S^n \times \mathbb{R})$,
\begin{equation*}
\cdots \leftarrow K_{-4} \leftarrow K_{-3} \leftarrow K_{-2} \leftarrow K_{-1} .
\end{equation*}

\item
We define $K_{-} :  = \mathrm{hocolim}_{i<0}K_i$. Note that $\mathrm{SS}(K_-)$ and $\Lambda_\phi$ coincide in $\dot{T}^*(S^n \times S^n \times \mathbb{R}_{<0})$.

\end{enumerate}

\begin{lem}
\label{lem3.7}
The object $ \Hom_{\Sh(S^n \times S^n \times \mathbb{R})}( K_{-}, K_{+}) $ is isomorphic to $ C^*(S^n) \otimes C^*(  \mathbb{R}^{n+1} , S^{n-1})[-1] $ in $\Mod(\mathbb{K})$.
\end{lem}
\begin{proof}
Set $A :  =\{(x,t) \in \mathbb{R}^n \times \mathbb{R} \mid |x| \leq -t \}$ and $ B :  =\{(x,t) \in \mathbb{R}^n \times \mathbb{R} \mid |x| < t \}$. 
We have the following chain of isomorphisms,
\begin{equation*}
\begin{split}
&\Hom_{\Sh(S^n \times S^n \times \mathbb{R})}(K_{-} , K_{+})  \\
&\cong \Hom_{\Sh(S^n \times S^n \times \mathbb{R})} (\mathrm{hocolim}_{i < 0} K_{i}, \mathrm{holim}_{j > 0} K_{j}) \\
&\cong \mathrm{holim}_{i < 0} \mathrm{holim}_{j > 0}  \Hom_{\Sh(S^n \times S^n \times \mathbb{R})} ( K_i,K_j) \\
&\cong \Hom_{\Sh(S^n \times S^n \times \mathbb{R})}(K_{-1}, K_{1}) \\
&\cong \Hom_{\Sh(S^n \times S^n \times \mathbb{R})}(\Cone(\mathbb{K}_{Z_{-2}} \xrightarrow{\psi_{-1}}\mathbb{K}_{Z_{-1}}), \Cocone(\mathbb{K}_{Z_{1}} \xrightarrow{\psi_1} \mathbb{K}_{Z_2})) \\
&\cong \Hom_{\Sh(S^n \times S^n \times \mathbb{R})}(\mathbb{K}_{Z_{-1}}, \mathbb{K}_{Z_{1}}) \\
&\cong \Hom_{\Sh(S^n \times \mathbb{R}^n \times \mathbb{R})}(\mathbb{K}_{S^n \times A}, \mathbb{K}_{S^n \times B}) \\
&\cong C^*(S^n) \otimes \Hom_{\Sh(\mathbb{R}^n \times \mathbb{R})}(\mathbb{K}_{A}, \mathbb{K}_{B}) \\
&\cong C^*(S^n) \otimes \Cocone( \Hom_{\Sh(\mathbb{R}^n \times \mathbb{R})}(\mathbb{K}_{\mathbb{R}^{n+1}}, \mathbb{K}_{B}) \to \Hom_{\Sh(\mathbb{R}^n \times \mathbb{R})}(\mathbb{K}_{\mathbb{R}^{n+1} \setminus A }, \mathbb{K}_{B})) \\
&\cong C^*(S^n) \otimes \Cocone( 0 \to \Hom_{\Sh(\mathbb{R}^n \times \mathbb{R})}(\mathbb{K}_{\mathbb{R}^{n+1}\setminus A}, \mathbb{K}_{B})) \\
&\cong C^*(S^n) \otimes \Hom_{\Sh(\mathbb{R}^n \times \mathbb{R})}(\mathbb{K}_{\mathbb{R}^{n+1} \setminus A}, \mathbb{K}_{B})[-1] \\
&\cong C^*(S^n) \otimes \Cocone( \Hom_{\Sh(\mathbb{R}^n \times \mathbb{R})}(\mathbb{K}_{\mathbb{R}^{n+1} \setminus A}, \mathbb{K}_{\mathbb{R}^{n+1}}) \to \Hom_{\Sh(\mathbb{R}^n \times \mathbb{R})}(\mathbb{K}_{\mathbb{R}^{n+1} \setminus A}, \mathbb{K}_{\mathbb{R}^{n+1} \setminus B}))[-1] \\
&\cong C^*(S^n) \otimes \Cocone( C^*(\mathbb{R}^{n+1} \setminus A) \to C^*( (\mathbb{R}^{n+1} \setminus A) \cap ( \mathbb{R}^{n+1} \setminus B))[-1] \\
&\cong C^*(S^n) \otimes C^*( \mathbb{R}^{n+1} \setminus A ,  (\mathbb{R}^{n+1} \setminus A) \cap ( \mathbb{R}^{n+1} \setminus B )  )  [-1] \\
&\cong C^*(S^n) \otimes C^*(\mathbb{R}^{n+1}, S^{n-1})[-1] .
\end{split}
\end{equation*}
The first isomorphism follows by the definition of $K_{-}$ and $K_{+}$.
The second follows from the fact that the $\Hom$-functor preserves homotopy limits.
The third follows from $ \Hom_{\Sh(S^n \times S^n \times \mathbb{R})} (  \mathbb{K}_{Z_i},  \mathbb{K}_{Z_j} ) \cong 0 $ if $ (i,j) \in \mathbb{Z}_{<0} \times \mathbb{Z}_{>0} \setminus \{(-1,1)\} $ and the following diagram:
\begin{equation*}
\begin{tikzcd}[column sep=tiny]
\  & \ \ar[d,"\cong"] & \ \ar[d,"\cong"] & \ \ar[d,"\cong"] \\
\ \ar[r,"\cong"] & \Hom_{\Sh(S^n \times S^n \times \mathbb{R})}(K_{-3},K_3) \ar[r,"\cong"] \ar[d,"\cong"] & \Hom_{\Sh(S^n \times S^n \times \mathbb{R})}(K_{-2},K_3) \ar[r,"\cong"] \ar[d,"\cong"] & \Hom_{\Sh(S^n \times S^n \times \mathbb{R})}(K_{-1},K_3) \ar[d,"\cong"]  \\
\ \ar[r,"\cong"] & \Hom_{\Sh(S^n \times S^n \times \mathbb{R})}(K_{-3},K_2) \ar[r,"\cong"] \ar[d,"\cong"] & \Hom_{\Sh(S^n \times S^n \times \mathbb{R})}(K_{-2},K_2) \ar[r,"\cong"] \ar[d,"\cong"] & \Hom_{\Sh(S^n \times S^n \times \mathbb{R})}(K_{-1},K_2) \ar[d,"\cong"]  \\
\ \ar[r,"\cong"] & \Hom_{\Sh(S^n \times S^n \times \mathbb{R})}(K_{-3},K_1) \ar[r,"\cong"] & \Hom_{\Sh(S^n \times S^n \times \mathbb{R})}(K_{-2},K_1) \ar[r,"\cong"] & \Hom_{\Sh(S^n \times S^n \times \mathbb{R})}(K_{-1},K_1) . \\ 
\end{tikzcd}
\end{equation*}
The fourth follows from the definition of $K_{-1}$ and $K_{1}$.
The fifth follows from $ \Hom_{\Sh(S^n \times S^n \times \mathbb{R})} (  \mathbb{K}_{Z_i},  \mathbb{K}_{Z_j} ) \cong 0 $ if $ (i,j) \in \mathbb{Z}_{<0} \times \mathbb{Z}_{>0} \setminus \{(-1,1)\} $.
The sixth follows from the coordinate transformation.
The seventh follows from the K\"{u}nneth formula.
The rest follow from Proposition \ref{prop2.5}, \ref{prop2.6}, \ref{prop2.7}, and the excision theorem.
\end{proof}

\begin{lem}
\label{lem3.8}
Set $ A :  ={\{(x,t) \in \mathbb{R}^n \times \mathbb{R} \mid |x| \leq -t\}}$, $ B :  =\{(x,t) \in \mathbb{R}^n \times \mathbb{R} \mid |x| < t\}$.
Let $ \theta \in \Hom^0_{\Sh(\mathbb{R}^n \times \mathbb{R})} ( \mathbb{K}_A , \mathbb{K}_B[n+1] ) $ be a morphism such that $H^0(\theta) \in \Ext^{n+1}(\mathbb{K}_A, \mathbb{K}_B) \cong \mathbb{K}$ is a generator. Then 
\begin{equation*}
\begin{split}
\mathrm{SS} ( \mathbb{K}_A ) \cap T^*_{(0,0)}(\mathbb{R}^n \times \mathbb{R}) = \{ a_1 dx_1 + \dots  + a_{n+1} dt \mid a_{n+1} \leq 0 , {a_1}^2 + \dots + {a_n}^2 \leq {a_{n+1}}^2 \}, \\
\mathrm{SS} (\mathbb{K}_B ) \cap T^*_{(0,0)}(\mathbb{R}^n \times \mathbb{R})= \{ a_1 dx_1 + \dots + a_{n+1} dt \mid a_{n+1} \leq 0 , {a_1}^2 + \dots + {a_n}^2 \leq {a_{n+1}}^2 \}, \\
\mathrm{SS} ( \Cocone(\theta) ) \cap T^*_{(0,0)}(\mathbb{R}^n \times \mathbb{R} )= \{ a_1 dx_1 + \dots + a_{n+1} dt \mid a_{n+1} \leq 0 , {a_1}^2 + \dots + {a_n}^2 = {a_{n+1}}^2 \}.
\end{split}
\end{equation*}
\end{lem}
\begin{proof}

The proofs for $\mathrm{SS}(\mathbb{K}_A)$ and $\mathrm{SS}(\mathbb{K}_B)$ are the same as Lemma \ref{lem3.4}.
We will prove the equality for $\mathrm{SS}(\Cocone(\theta))$. 
Consider $\xi= a_1d x_1+\dots+a_{n+1}dt  \in T^*_{(0,0)}(\mathbb{R}^n \times \mathbb{R} )$ and a $C^\infty$-function $f\colon \mathbb{R}^n \times \mathbb{R} \to \mathbb{R}$ such that  $f(0,0)=0$ and $d_{(0,0)}f=\xi$. 
Let $i \colon  \{f \geq 0\} \to \mathbb{R}^n \times \mathbb{R}$ be the inclusion map.
If $a_{n+1} < 0 $ and $ {a_1}^2 + \dots + {a_n}^2 < {a_{n+1}}^2$, then there exists a sufficiently small $\epsilon>0$ such that $A \cap B(0;\epsilon) \subset \{f \geq 0\} \cap B(0;\epsilon) \subset B^c \cap B(0;\epsilon)$ and we have the following commutative diagram: 

\begin{equation*}
\begin{tikzcd}
(\Gamma_{\{f \geq 0 \}}\mathbb{K}_A)_{(0,0)}  \ar[r] \ar[d,"\cong"] & (\Gamma_{\{f \geq 0 \}}\mathbb{K}_B[n+1])_{(0,0)}  \ar[d,"\cong"] \\
\Gamma( B(0;\epsilon) ; \Gamma_{\{f \geq 0 \}}\mathbb{K}_A ) \ar[d,"\cong"]  & \Gamma( B(0;\epsilon) ; \Gamma_{\{f \geq 0 \}}\mathbb{K}_B[n+1])  \ar[d,"\cong"] \\
\Hom_{\Sh(\mathbb{R}^n \times \mathbb{R})}(\mathbb{K}_{ B(0;\epsilon) } , \Gamma_{ \{f \geq 0 \} }\mathbb{K}_A)  \ar[d,"\cong"]  & \Hom_{\Sh(\mathbb{R}^n \times \mathbb{R})}(\mathbb{K}_{ B(0;\epsilon) } , \Gamma_{\{ f \geq 0 \}} \mathbb{K}_B[n+1]) \ar[d,"\cong"] \\ 
\Hom_{\Sh(\mathbb{R}^n \times \mathbb{R})}(\mathbb{K}_{ B(0;\epsilon) } , i_*i^{-1}\mathbb{K}_A) \ar[d,"\cong"]  &  \Hom_{\Sh(\mathbb{R}^n \times \mathbb{R})}(\mathbb{K}_{B(0;\epsilon)} , i_*i^! \mathbb{K}_B[n+1] ) \ar[d,"\cong"] \\
\Hom_{\Sh(\mathbb{R}^n \times \mathbb{R})}(i_!i^{-1}\mathbb{K}_{ B(0;\epsilon) } , \mathbb{K}_A) \ar[d,"\cong"]  &  \Hom_{\Sh(\mathbb{R}^n \times \mathbb{R})}(i_!i^{-1}\mathbb{K}_{ B(0;\epsilon) } ,\mathbb{K}_B[n+1] )  \ar[d,"\cong"] \\
\Hom_{\Sh(\mathbb{R}^n \times \mathbb{R})}(\mathbb{K}_{  \{f \geq 0\}\cap B(0;\epsilon) } , \mathbb{K}_A) \ar[d,"\cong"]  &  \Hom_{\Sh(\mathbb{R}^n \times \mathbb{R})}(\mathbb{K}_{ \{f \geq 0\}\cap B(0;\epsilon) } ,\mathbb{K}_B[n+1] )  \ar[d,"\cong"] \ar[d,"\cong"] \\
\Hom_{\Sh(\mathbb{R}^n \times \mathbb{R})}(\mathbb{K}_A, \mathbb{K}_A) \ar[r,"\theta \circ -"]  & \Hom_{\Sh(\mathbb{R}^n \times \mathbb{R})}(\mathbb{K}_A, \mathbb{K}_B[n+1]).
\end{tikzcd}
\end{equation*}
By the definition of $\theta$, the morphism $(\Gamma_{\{f \geq 0 \}} \mathbb{K}_A)_{(0,0)} \to (\Gamma_{\{ f \geq 0 \}}\mathbb{K}_B[n+1])_{(0,0)}$ is an isomorphism. 
Therefore we have $\xi \notin \mathrm{SS}(\Cocone(\theta))$ and
\begin{equation*}
\mathrm{SS} ( \Cocone(\theta) ) \cap T^*_{(0,0)}(\mathbb{R}^n \times \mathbb{R}) \subset \{ a_1 dx_1 + \dots + a_{n+1} dt \mid a_{n+1} \leq 0 , {a_1}^2 + \dots + {a_n}^2 = {a_{n+1}}^2 \}.
\end{equation*}
The converse
\begin{equation*}
\mathrm{SS} ( \Cocone(\theta) ) \cap T^*_{(0,0)}(\mathbb{R}^n \times \mathbb{R}) \supset \{ a_1 dx_1 + \dots + a_{n+1} dt \mid a_{n+1} \leq 0 , {a_1}^2 + \dots + {a_n}^2 = {a_{n+1}}^2 \}
\end{equation*}
follows by a direct computation.
\end{proof}

We choose a morphism $\psi_0 \in \Hom^0_{\Sh(S^n \times S^n \times \mathbb{R})} ( K_{-} , K_{+}[n+1] ) $ such that $H^0(\psi_0) \in \Ext^{n+1}(K_{-}, K_{+}) \cong \mathbb{K}$ is a generator. Note that the isomorphism $\Ext^{n+1}(K_{-}, K_{+}) \cong \mathbb{K}$ follows from Lemma \ref{lem3.7}. 

\begin{prop}
\label{prop3.9}
There is an equality,
\begin{equation*}
\mathrm{SS} (\Cocone(\psi_0)) \cap \dot{T}^*_{(x,x,0)} (S^n \times S^n \times \mathbb{R} ) = \Lambda_\phi \cap \dot{T}^*_{(x,x,0)}(S^n \times S^n \times \mathbb{R}).
\end{equation*}
\end{prop}
\begin{proof}
It follows from Lemma \ref{lem3.8}.
\end{proof}

\begin{thm}
The sheaf $K:=\Cocone(\psi_0)$ is the GKS kernel of $\phi$. 
\end{thm}
\begin{proof}
We have to prove that $\mathrm{SS}(K) \setminus 0_{S^n \times S^n \times \mathbb{R}} = \Lambda_\phi$ and $K|_{ \{ t=0 \} } \cong \mathbb{K}_{\Delta_{S^n}}$. The first follows from Proposition \ref{prop3.5}, \ref{prop3.6}, and \ref{prop3.9}. The second follows from the following chain of isomorphisms, 
\begin{equation*}
\begin{split}
K|_{\{t=0\}} 
&\cong \Cocone( K_{-}|_{\{t=0\}} \to K_{+}[n+1]|_{\{t=0\}} ) \\
&\cong \Cocone( K_{-}|_{\{t=0\}} \to 0) \\
&\cong K_{-}|_{\{t=0\}} \\
&\cong K_{-1}|_{\{t=0\}} \\
&\cong \Cone(\mathbb{K}_{Z_{-2}}[-n]|_{\{t=0\}} \to \mathbb{K}_{Z_{-1}}|_{\{t=0\}}) \\
&\cong \Cone(0 \to \mathbb{K}_{Z_{-1}}|_{\{t=0\}}) \\
&\cong \mathbb{K}_{Z_{-1}}|_{\{t=0\}} \\
&\cong \mathbb{K}_{\Delta_{S^n}} .
\end{split}
\end{equation*}
\end{proof}

\newpage

\section{GKS kernel of the geodesic flow on $\mathbb{CP}^n$}

We fix the Fubini--Study metric $g$ on $\mathbb{CP}^n$, that is, $g$ is the metric induced by the K\"{a}hler form $ \omega=\sum_{i,j=1}^{n} 2{\sqrt{-1}} h_{ij}{dz_i} \wedge d{\bar{z}_j} $ where 
\begin{equation*}
h_{ij}=\frac{(1+ \sum_{k=1}^n |z_k|^2 ) \delta_{ij}-\bar{z}_i z_j}{(1+ \sum_{k=1}^n |z_k|^2)^2}.
\end{equation*}
We consider the Hamiltonian function $H\colon \dot{T}^*\mathbb{CP}^n \to \mathbb{R} , H(x,\xi)=|\xi|_{g^*}$ and the induced Hamiltonian isotopy $\phi \colon  \dot{T}^*\mathbb{CP}^n \times \mathbb{R} \to \dot{T}^*\mathbb{CP}^n$. In this section, we shall explicitly construct the GKS kernel $K_\phi$ of the geodesic flow $\phi$. \\

\begin{lem}
\label{lem4.1}
Set $x_0 : =(1 \colon  0 \colon  \dots \colon  0) \in \mathbb{CP}^n$. Then the exponential map from a point $x_0$ by time $t$ is given by $ \Exp_{x_0, t} \colon  S(T_{x_0} \mathbb{CP}^n) \to \mathbb{CP}^n $, $ (z_1,\dots , z_n) \mapsto (\cos \frac{t}{2} \colon  z_1 \sin \frac{t}{2} \colon  \dots \colon  z_n \sin \frac{t}{2} ) $ where $S(T_{x_0}\mathbb{CP}^n)$ is the unit sphere of $T_{x_0}\mathbb{CP}^n$ and we identify $S(T_{x_0}\mathbb{CP}^n )$ with $ \{ (z_1,\dots,z_n) \in \mathbb{C}^n \mid |z_1|^2 + \cdots + |z_n|^2 =1 \} $ by the affine coordinate.
In particular, the diameter of $\mathbb{CP}^n$ is $\pi$ and the period of the geodesic flow on $\mathbb{CP}^n$ is $2\pi$.
\end{lem}
\begin{proof}
Let $J \in \Gamma(T\mathbb{CP}^n \otimes T^* \mathbb{CP}^n)$ be the complex structure on $\mathbb{CP}^n$. Set $U_0: =\{(X_0 \colon  \dots \colon  X_n) \in \mathbb{CP}^n \mid X_0 \neq 0 \} $. Fix a point $(z_1, \dots , z_n) \in \mathbb{C}^n$ which satisfies $|z_1|^2 + \cdots + |z_n|^2 =1$. Consider the map $c \colon  (-\pi,\pi) \to U_0 $, $c(t)=( z_1 \tan \frac{t}{2} , \dots ,  z_n \tan \frac{t}{2}) $ where we identify $ U_0 $ and $\mathbb{C}^n$ by the affine coordinate. Then we have
\begin{equation*}
\frac{dc}{dt}(t) = \frac{z_1}{2\cos^2 \frac{t}{2}} \frac{\partial}{\partial z_1} +  \frac{\bar{z}_1}{2\cos^2 \frac{t}{2}} \frac{\partial}{\partial \bar{z}_1} + \cdots + \frac{z_n}{2\cos^2 \frac{t}{2}} \frac{\partial}{\partial z_n} +  \frac{\bar{z}_n}{2\cos^2 \frac{t}{2}} \frac{\partial}{\partial \bar{z}_n} \in T_{c(t)}{U_0} 
\end{equation*}
and
\begin{equation*}
J(\frac{dc}{dt}(t)) = \frac{\sqrt{-1}z_1}{2\cos^2 \frac{t}{2}} \frac{\partial}{\partial z_1} -  \frac{\sqrt{-1}\bar{z}_1}{2\cos^2 \frac{t}{2}} \frac{\partial}{\partial \bar{z}_1} + \cdots + \frac{\sqrt{-1}z_n}{2\cos^2 \frac{t}{2}} \frac{\partial}{\partial z_n} - \frac{\sqrt{-1}\bar{z}_n}{2\cos^2 \frac{t}{2}} \frac{\partial}{\partial \bar{z}_n} \in T_{c(t)}{U_0} .
\end{equation*}
Therefore
\begin{equation*}
\begin{split}
g(\frac{dc}{dt}(t),\frac{dc}{dt}(t))&=\omega(\frac{dc}{dt}(t), J(\frac{dc}{dt}(t))) \\
&=\sum_{i,j=1}^n 2{\sqrt{-1}} h_{ij} dz_i \wedge d\bar{z}_j(\frac{dc}{dt}(t), J(\frac{dc}{dt}(t)))\\
&=\sum_{i,j=1}^n 2{\sqrt{-1}} h_{ij} dz_i \wedge d\bar{z}_j(\frac{z_i}{2\cos^2 \frac{t}{2}}\frac{\partial}{\partial z_i}+\frac{\bar{z}_j}{2\cos^2\frac{t}{2}}\frac{\partial}{\partial \bar{z}_j}, \frac{\sqrt{-1}z_i}{2\cos^2 \frac{t}{2}}\frac{\partial}{\partial z_i}-\frac{\sqrt{-1}\bar{z}_j}{2\cos^2 \frac{t}{2}}\frac{\partial}{\partial \bar{z}_j} ) \\
&=\sum_{i,j=1}^n 2{\sqrt{-1}} h_{ij}  \frac{- 2\sqrt{-1} z_i \bar{z}_j}{(2\cos^2 \frac{t}{2})^2} \\
&=\sum_{i,j=1}^n \frac{ z_i \bar{z}_j  } {(\cos^2 \frac{t}{2})^2}  h_{ij} \\
&=\sum_{i,j=1}^n \frac{z_i \bar{z}_j}{(\cos^2 \frac{t}{2})^2}\frac{(1+ \sum_{k=1}^n | z_k \tan \frac{t}{2}|^2 ) \delta_{ij} - (\bar{z}_i \tan \frac{t}{2} ) ( z_j \tan \frac{t}{2}) }{(1+ \sum_{k=1}^n |z_k \tan \frac{t}{2}|^2)^2} \\
&=\sum_{i,j=1}^n \frac{ z_i \bar{z}_j  }{(\cos^2 \frac{t}{2})^2} \frac{ (1+\tan^2 \frac{t}{2}) \delta_{ij} - \bar{z}_i z_j \tan^2 \frac{t}{2}}{ (1+\tan^2 \frac{t}{2})^2 } \\
&=\sum_{i,j=1}^n { z_i \bar{z}_j  } \{ (1+\tan^2 \frac{t}{2}) \delta_{ij} - \bar{z}_i z_j \tan^2 \frac{t}{2} \} \\
&=\sum_{i=1}^n { |z_i|^2  (  1+ \tan^2 \frac{t}{2} - |z_i|^2 \tan^2 \frac{t}{2})} - \sum_{i=1}^n \sum_{j \neq i}{|z_i|^2 |z_j|^2 \tan^2 \frac{t}{2} } \\
&=\sum_{i=1}^n { |z_i|^2  (  1+ \tan^2 \frac{t}{2} - |z_i|^2 \tan^2 \frac{t}{2}- \sum_{j \neq i}  |z_j|^2 \tan^2 \frac{t}{2} ) } \\
&=\sum_{i=1}^n { |z_i|^2  (  1+ \tan^2 \frac{t}{2} - \tan^2 \frac{t}{2} ) } \\
&=1.
\end{split}
\end{equation*}
Moreover, $ \mathrm{Im}(c) \subset U_0 $ is a straight line in $U_0 \cong \mathbb{C}^n$. Therefore the map $c\colon  (-\pi,\pi) \to U_0$ is the unique geodesic on $U_0$ such that $c(0)=x_0$ and $\frac{dc}{dt}(0) = (z_1,\dots,z_n) $ where we use the identification $S(T_{x_0}U_0) \cong \{(z_1,\dots,z_n) \in \mathbb{C}^n \mid |z_1|^2 + \cdots + |z_n|^2 =1 \}$ by the affine coordinate. 

Consider the map $\tilde{c} \colon \mathbb{R} \to \mathbb{CP}^n$, $\tilde{c}(t)=(\cos \frac{t}{2}\colon  z_1 \sin \frac{t}{2} \colon  \dots \colon  z_n \sin \frac{t}{2})$. Note that $\tilde{c}$ is an extension of $c$. Since $\tilde{c}|_{\mathbb{R} \setminus \{ \pm\pi, \pm3\pi, \pm5\pi, \dots \} }  \colon  \mathbb{R} \setminus \{ \pm \pi , \pm 3\pi , \pm 5 \pi, \dots \} \to \mathbb{CP}^n$ is the geodesic and $\tilde{c} \colon  \mathbb{R} \to \mathbb{CP}^n$ is smooth, we conclude that $\tilde{c}$ is the geodesic on $\mathbb{CP}^n$ with initial velocity $(z_1,\dots ,z_n) \in S( T_{x_0}\mathbb{CP}^n )$. In particular, $\Exp_{x_0,t}(z_1,\dots,z_n) = \tilde{c}(t)$ for any $t \in \mathbb{R}$.
\end{proof}

For $x \in \mathbb{CP}^n$, we define the subset $D_{x} \subset \mathbb{CP}^n$ by $D_{x} = \{y \in \mathbb{CP}^n \mid \mathrm{dist}(x,y)=\pi \} = \{ \Exp_{x,\pi}(v) \in \mathbb{CP}^n \mid v \in S(T_{x}{\mathbb{CP}^n}) \}$. Note that $D_{x}$ is diffeomorphic to $\mathbb{CP}^{n-1}$ and the exponential map $ \Exp_{x,\pi} \colon  S(T_{x} \mathbb{CP}^n) \to D_{x}$ is an $S^1$-fibration.

For $x \in \mathbb{CP}^n$ and $y \in D_{x}$,  we define the subset $D_{x, y} \subset \mathbb{CP}^n$ by $D_{x,y} = \{ \Exp_{x,t}(v) \in \mathbb{CP}^n \mid v \in S(T_{x}\mathbb{CP}^n), \Exp_{x,\pi}(v) = y , t \in \mathbb{R} \}$. Note that $D_{x, y}$ is diffeomorphic to $\mathbb{CP}^1$. 

For $i \in \mathbb{Z}_{>0}$ odd, we define the open subset $Z_i$ of $\mathbb{CP}^n \times \mathbb{CP}^n \times \mathbb{R}$ by
\begin{equation*}
    Z_i:=\{(x,y,t)\in \mathbb{CP}^n \times \mathbb{CP}^n \times \mathbb{R} \mid \mathrm{dist}(x,y)<t-(i-1)\pi\}.
\end{equation*}

For $i \in \mathbb{Z}_{>0}$ even, we define the open subset $Z_i$ of $\mathbb{CP}^n \times \mathbb{CP}^n \times \mathbb{R}$ by
\begin{equation*}
    Z_i: =\{(x,y,t)\in \mathbb{CP}^n \times \mathbb{CP}^n \times \mathbb{R} \mid \mathrm{dist}(x,D_y)<t-(i-1)\pi\}.
\end{equation*}

\begin{lem}
\label{lem4.2}
For $i \in \mathbb{Z}_{>0}$ and $\epsilon > 0$, the inclusion map $ j \colon  \mathbb{CP}^n \times \mathbb{CP}^n \times \{i\pi +\epsilon \} \to Z_i$ is a homotopy equivalence.
\end{lem}
\begin{proof}
The natural map $r\colon  Z_i \to \mathbb{CP}^n \times \mathbb{CP}^n \times \{ i\pi + \epsilon \}$, $r (x,y,t) = (x, y, i\pi + \epsilon)$ is a homotopy equivalence because a homotopy $H\colon Z_i \times [0,1] \to Z_i$ between $\mathrm{id}_{Z_i}$ and $j \circ r$ is given by $H(x,y,t,s)=(x,y,(1-s)t + s(i\pi + \epsilon) )$.
\end{proof}

\begin{lem}
\label{lem4.3}
For $i \in \mathbb{Z}_{>0}$ odd and $0< \epsilon < \pi$, the inclusion map $j \colon  \Delta_{\mathbb{CP}^n} \times \{i\pi +\epsilon \} \to Z_i \setminus Z_{i+1}$ is a homotopy equivalence. For $i \in \mathbb{Z}_{>0}$ even and $0<\epsilon < \pi$, the inclusion map $j\colon  \bigcup_{x \in \mathbb{CP}^n} \{x\} \times D_x \times \{ i \pi + \epsilon \} \to Z_i \setminus Z_{i+1}$ is a homotopy equivalence.
\begin{proof}
Consider the case of $i$ odd. For $(x,y) \notin \Delta_{\mathbb{CP}^n} \cup \bigcup_{x \in \mathbb{CP}^n} \{x\} \times D_x$, we define $v_x\in T_y{\mathbb{CP}^n}$ by the equations $\Exp_{y,\mathrm{dist}(x,y)}(v_x) = x$ and $|v_x|_g=1$. For $(x,y) \in \Delta_{\mathbb{CP}^n}$, we define $v_x \in T_y{\mathbb{CP}^n}$ by $v_x=0$. Note that $v_x$ is determined uniquely. The natural map $r\colon  Z_i \setminus Z_{i+1} \to \Delta_{\mathbb{CP}^n} \times \{ i\pi + \epsilon \}$, $r (x,y,t) = (x, x, i\pi + \epsilon)$ is a homotopy equivalence because a homotopy $H\colon Z_i\setminus Z_{i+1} \times [0,1] \to Z_i \setminus Z_{i+1}$ between $\mathrm{id}_{Z_i \setminus Z_{i+1}}$ and $j \circ r$ is given by $H(x,y,t,s)=(x, \Exp_{y,s \mathrm{dist}(x,y)}(v_x) ,(1-s)t + s(i\pi + \epsilon) )$.

Consider the case of $i$ even. For $(x,y) \notin \Delta_{\mathbb{CP}^n} \cup \bigcup_{x \in \mathbb{CP}^n} \{x\} \times D_x$, we define $v_x\in T_y{\mathbb{CP}^n}$ by the equations $\Exp_{y,\mathrm{dist}(x,y)}(v_x) = x$ and $|v_x|_g=1$. For $(x,y) \in \bigcup_{x \in \mathbb{CP}^n} \{x\} \times D_x$, we define $v_x \in T_y{\mathbb{CP}^n}$ by $v_x=0$. Note that $v_x$ is determined uniquely. The natural map $r\colon  Z_i \setminus Z_{i+1} \to \bigcup_{x \in \mathbb{CP}^n} \{x\} \times D_x \times \{ i\pi + \epsilon \}$, $r (x,y,t) = (x, \Exp_{y,\pi - \mathrm{dist}(x,y)}(-v_x), i \pi + \epsilon)$ is a homotopy equivalence because a homotopy $H\colon Z_i\setminus Z_{i+1} \times [0,1] \to Z_i \setminus Z_{i+1}$ between $\mathrm{id}_{Z_i \setminus Z_{i+1}}$ and $j \circ r$ is given by $H(x,y,t,s)=(x, \Exp_{y,s(\pi-\mathrm{dist}(x,y)) }(-v_x) ,(1-s)t + s(i\pi + \epsilon) )$.
\end{proof}
\end{lem}

\begin{lem}
\label{lem4.4}
For $i \in \mathbb{Z}_{>0}$ odd, the object $ \Hom_{\Sh(\mathbb{CP}^n \times \mathbb{CP}^n \times \mathbb{R}) }(\mathbb{K}_{Z_i}, \mathbb{K}_{Z_{i+1}}) $ is isomorphic to $ C^*(\mathbb{CP}^n \times \mathbb{CP}^n , \mathbb{CP}^n \times \mathbb{CP}^n \setminus \bigcup_{x \in \mathbb{CP}^n } \{x\} \times D_x ) $ in $\Mod(\mathbb{K})$.
For $i \in \mathbb{Z}_{>0}$ even, the object $\Hom_{\Sh(\mathbb{CP}^n \times \mathbb{CP}^n \times \mathbb{R}) }(\mathbb{K}_{Z_i}, \mathbb{K}_{Z_{i+1}}) $ is isomorphic to $ C^*(\mathbb{CP}^n \times \mathbb{CP}^n , \mathbb{CP}^n \times \mathbb{CP}^n \setminus \Delta_{\mathbb{CP}^n})$ in $\Mod(\mathbb{K})$.
\end{lem}
\begin{proof}
Consider the case of $i$ odd. We have the following chain of isomorphisms,
\begin{equation*}
\begin{split}
&\Hom_{\Sh(\mathbb{CP}^n \times \mathbb{CP}^n \times \mathbb{R}) }(\mathbb{K}_{Z_i},\mathbb{K}_{Z_{i+1}}) \\
&\cong \Cocone( \Hom_{\Sh(\mathbb{CP}^n \times \mathbb{CP}^n \times \mathbb{R}) }(\mathbb{K}_{Z_i}, \mathbb{K}_{\mathbb{CP}^n \times \mathbb{CP}^n \times \mathbb{R} }) \to \Hom_{\Sh(\mathbb{CP}^n \times \mathbb{CP}^n \times \mathbb{R} ) }(\mathbb{K}_{Z_i} , \mathbb{K}_{{Z_{i+1}}^c})) \\
&\cong \Cocone(C^*(Z_i) \to C^*(Z_i \setminus Z_{i+1})) \\
&\cong \Cocone(C^*(\mathbb{CP}^n \times \mathbb{CP}^n) \to C^*(\Delta_{\mathbb{CP}^n} ) ) \\
&\cong C^*(\mathbb{CP}^n \times \mathbb{CP}^n , \Delta_{\mathbb{CP}^n} ) \\
&\cong C^*(\mathbb{CP}^n \times \mathbb{CP}^n , \mathbb{CP}^n \times \mathbb{CP}^n \setminus \bigcup_{x \in \mathbb{CP}^n } \{x\} \times D_x ).
\end{split}
\end{equation*}
The first isomorphism follows from Proposition \ref{prop2.5}.
The second follows from Proposition \ref{prop2.6}.
The third follows from Lemma \ref{lem4.2} and \ref{lem4.3}.
The fourth follows from Proposition \ref{prop2.7}.
The fifth follows from the fact that the inclusion map $\Delta_{\mathbb{CP}^n} \to \mathbb{CP}^n \times \mathbb{CP}^n \setminus \bigcup_{x \in \mathbb{CP}^n } \{x\} \times D_x $ is a homotopy equivalence.

Consider the case of $i$ even. We have the following chain of isomorphisms,
\begin{equation*}
\begin{split}
&\Hom_{\Sh(\mathbb{CP}^n \times \mathbb{CP}^n \times \mathbb{R}) }(\mathbb{K}_{Z_i}, \mathbb{K}_{Z_{i+1}}) \\
&\cong \Cocone( \Hom_{\Sh(\mathbb{CP}^n \times \mathbb{CP}^n \times \mathbb{R}) }(\mathbb{K}_{Z_i}, \mathbb{K}_{\mathbb{CP}^n \times \mathbb{CP}^n \times \mathbb{R} }) \to \Hom_{\Sh(\mathbb{CP}^n \times \mathbb{CP}^n \times \mathbb{R}) }(\mathbb{K}_{Z_i} , \mathbb{K}_{{Z_{i+1}}^c})) \\
&\cong \Cocone(C^*(Z_i) \to C^*(Z_i \setminus Z_{i+1})) \\
&\cong \Cocone(C^*(\mathbb{CP}^n \times \mathbb{CP}^n) \to C^*( \bigcup_{x \in \mathbb{CP}^n} \{ x \} \times D_{x})) \\
&\cong C^*(\mathbb{CP}^n \times \mathbb{CP}^n , \bigcup_{x \in \mathbb{CP}^n} \{x\} \times D_x) \\
&\cong C^*( \mathbb{CP}^n \times \mathbb{CP}^n , \mathbb{CP}^n \times \mathbb{CP}^n \setminus \Delta_{\mathbb{CP}^n}).
\end{split}
\end{equation*}
The first isomorphism follows from Proposition \ref{prop2.5}. The second follows from Proposition \ref{prop2.6}. The third follows from Lemma \ref{lem4.2} and \ref{lem4.3}. The fourth follows from Proposition \ref{prop2.7}. The fifth follows from the fact that the inclusion map $ \bigcup_{x \in \mathbb{CP}^n} \{x\} \times D_x \to \mathbb{CP}^n \times \mathbb{CP}^n \setminus \Delta_{\mathbb{CP}^n}$ is a homotopy equivalence. 
\end{proof}

\begin{lem}[{\cite[Lemma 11.7]{MS}}]
\label{lem4.5}
For a point $x \in \mathbb{CP}^n$, we define the map $j_{x} \colon  \mathbb{CP}^n \to \mathbb{CP}^n \times \mathbb{CP}^n$ by $j_{x}(y)=(x,y)$. Then there exists a unique cohomology class $u \in H^{2n}(\mathbb{CP}^n \times \mathbb{CP}^n , \mathbb{CP}^n \times \mathbb{CP}^n \setminus \Delta_{\mathbb{CP}^n})$ such that for any $x \in \mathbb{CP}^n$ the map $ {j_{x}}^* \colon  H^{2n}(\mathbb{CP}^n \times \mathbb{CP}^n , \mathbb{CP}^n \times \mathbb{CP}^n \setminus \Delta_{\mathbb{CP}^n}) \to H^{2n}(  \mathbb{CP}^n , \mathbb{CP}^n \setminus \{ x \} )$ sends $u$ to a generator of $H^{2n}(\mathbb{CP}^n ,  \mathbb{CP}^n \setminus \{ x \} ) \cong \mathbb{K}$.
\end{lem}

For $i \in \mathbb{Z}_{>0}$ odd, we choose a morphism $\psi_i \in \Hom^0_{\Sh(\mathbb{CP}^n \times \mathbb{CP}^n \times \mathbb{R})} (\mathbb{K}_{Z_i}, \mathbb{K}_{Z_{i+1}}[2]) $ such that $ H^0(\psi_i) \in \Ext^2(\mathbb{K}_{Z_i}, \mathbb{K}_{Z_{i+1}}) \cong \mathbb{K}$ is a generator. Note that the isomorphism $\Ext^2(\mathbb{K}_{Z_i}, \mathbb{K}_{Z_{i+1}}) \cong \mathbb{K}$ follows from Lemma \ref{lem4.4}.
For $i \in \mathbb{Z}_{>0}$ even, we choose a morphism $\psi_i \in \Hom^0_{\Sh(\mathbb{CP}^n \times \mathbb{CP}^n \times \mathbb{R})} (\mathbb{K}_{Z_i}, \mathbb{K}_{Z_{i+1}}[2n]) $ such that $ H^0(\psi_i) \in \Ext^{2n}(\mathbb{K}_{Z_i}, \mathbb{K}_{Z_{i+1}} ) \cong \mathbb{K}$ is a generator. Note that the isomorphism $ \Ext^{2n}(\mathbb{K}_{Z_i}, \mathbb{K}_{Z_{i+1}} ) \cong \mathbb{K} $ follows from Lemma \ref{lem4.4} and the Thom isomorphism. Thus we obtain the sequence in $ \Sh(\mathbb{CP}^n \times \mathbb{CP}^n \times \mathbb{R})$,
\begin{equation*}
\mathbb{K}_{Z_1} \xrightarrow{\psi_1} \mathbb{K}_{Z_2}[2] \xrightarrow{\psi_2} \mathbb{K}_{Z_3}[2n+2] \xrightarrow{\psi_3} \mathbb{K}_{Z_4}[2n+4] \xrightarrow{\psi_4} \cdots .
\end{equation*}

\begin{prop}
\label{prop4.6}
For $i \in \mathbb{Z}_{>0}$ odd, there is an equality,
\begin{equation*}
\begin{split}
\mathrm{SS} (\mathbb{K}_{Z_i}) \cap \dot{T}^*_{(x,x,(i-1)\pi)}(\mathbb{CP}^n \times \mathbb{CP}^n \times \mathbb{R})&= \left\{ \sum_k \xi_k \middle| \xi_k \in \Lambda_\phi \cap \dot{T}^*_{(x,x,(i-1)\pi)}(\mathbb{CP}^n \times \mathbb{CP}^n \times \mathbb{R})  \right\}, \\
\mathrm{SS} (\mathbb{K}_{Z_i}) \cap \dot{T}^*_{(x,y,i\pi)}(\mathbb{CP}^n \times \mathbb{CP}^n \times \mathbb{R}) &= \left\{ \sum_k \xi_k \middle| \xi_k \in \Lambda_\phi \cap \dot{T}^*_{(x,y,i\pi)}(\mathbb{CP}^n \times \mathbb{CP}^n \times \mathbb{R}) \right\} \ \text{(if $y \in D_x$)}, \\
\mathrm{SS} (\mathbb{K}_{Z_i}) \cap \dot{T}^*_{(x,y,t)}(\mathbb{CP}^n \times \mathbb{CP}^n \times \mathbb{R}) &= \Lambda_{\phi} \cap \dot{T}^*_{(x,y,t)}(\mathbb{CP}^n \times \mathbb{CP}^n \times \mathbb{R} ) \ \text{(if $(i-1)\pi<t<i\pi$)}, \\
\mathrm{SS} (\mathbb{K}_{Z_i}) \cap \dot{T}^*_{(x,y,t)}(\mathbb{CP}^n \times \mathbb{CP}^n \times \mathbb{R}) &= \emptyset \ \text{(otherwise)}. \\
\end{split}
\end{equation*}
For $i \in \mathbb{Z}_{>0}$ even, there is an equality,
\begin{equation*}
\begin{split}
\mathrm{SS} (\mathbb{K}_{Z_i}) \cap \dot{T}^*_{(x,y,(i-1)\pi)}(\mathbb{CP}^n \times \mathbb{CP}^n \times \mathbb{R})&= \left\{ \sum_k \xi_k \middle| \xi_k \in \Lambda_\phi \cap \dot{T}^*_{(x,y,(i-1)\pi)}(\mathbb{CP}^n \times \mathbb{CP}^n \times \mathbb{R})  \right\}  \ \text{(if $y \in D_x$)}, \\
\mathrm{SS} (\mathbb{K}_{Z_i}) \cap \dot{T}^*_{(x,x,i\pi)}(\mathbb{CP}^n \times \mathbb{CP}^n \times \mathbb{R}) &= \left\{ \sum_k \xi_k \middle| \xi_k \in \Lambda_\phi \cap \dot{T}^*_{(x,x,i\pi)}(\mathbb{CP}^n \times \mathbb{CP}^n \times \mathbb{R}) \right\}, \\
\mathrm{SS} (\mathbb{K}_{Z_i}) \cap \dot{T}^*_{(x,y,t)}(\mathbb{CP}^n \times \mathbb{CP}^n \times \mathbb{R}) &= \Lambda_{\phi} \cap \dot{T}^*_{(x,y,t)}(\mathbb{CP}^n \times \mathbb{CP}^n \times \mathbb{R} ) \ \text{(if $(i-1)\pi<t<i\pi$)}, \\
\mathrm{SS} (\mathbb{K}_{Z_i}) \cap \dot{T}^*_{(x,y,t)}(\mathbb{CP}^n \times \mathbb{CP}^n \times \mathbb{R}) &= \emptyset \ \text{(otherwise)}. \\
\end{split}
\end{equation*}
\end{prop}
\begin{proof}
It follows from the definition of $Z_i$ and Lemma \ref{lem3.4}.
\end{proof}

\begin{prop}
\label{prop4.7}
For $i \in \mathbb{Z}_{>0}$ odd, there is an equality, 
\begin{equation*}
\mathrm{SS}(\Cocone(\psi_i)) \cap \dot{T}^*_{(x,y,i\pi)}(\mathbb{CP}^n \times \mathbb{CP}^n \times \mathbb{R}) = \Lambda_\phi \cap \dot{T}^*_{(x,y,i\pi)}(\mathbb{CP}^n \times \mathbb{CP}^n \times \mathbb{R}) \ \text{(if $y \in D_x$)}.
\end{equation*}
\end{prop}
\begin{proof}
We only prove the case of $i=1$ since the proof for $i > 1$ is similar.
Fix points $x_0 \in \mathbb{CP}^n$ and $ y_0 \in D_{x_0} $.
Set $W  :  = \{x_0\} \times \mathbb{CP}^n \times (\pi-\epsilon, \pi+\epsilon)$, $W'  :  = \{x_0\} \times D_{x_0,y_0} \times (\pi-\epsilon, \pi+\epsilon)$, and $W''  : = \{x_0\} \times ( D_{x_0,y_0} \cap B(y_0;\epsilon) ) \times (\pi-\epsilon, \pi+\epsilon)$. 
Then it suffices to prove the equality $\mathrm{SS} (\Cocone(\psi_1 |_{W})) \cap \dot{T}^*_{(x_0,y_0,\pi)}{W} = \{((y_0,\xi) , (\pi, -|\xi|_{g^*})) \in \dot{T}^*_{y_0}\mathbb{CP}^n \times T^*_{\pi}\mathbb{R}  \mid (y_0, \xi) \in \dot{T}^*_{y_0}\mathbb{CP}^n  \}$.
Since the microsupport is a local invariant, it suffices to work on a sufficiently small trivializing neighborhood of the $S^1$-fibration $\Exp_{x_0,\pi}\colon S(T_{x_0}\mathbb{CP}^n) \to D_{x_0}$.
Hence, it suffices to prove the equality $\mathrm{SS} (\Cocone(\psi_1 |_{W''})) \cap \dot{T}^*_{(x_0,y_0,\pi)}{W''} = \{((y_0,\xi) , (\pi, -|\xi|_{g^*})) \in \dot{T}^*_{y_0}{D_{x_0,y_0}} \times T^*_{\pi}\mathbb{R}  \mid (y_0, \xi) \in \dot{T}^*_{y_0}D_{x_0,y_0}  \}$.
We have the following commutative diagram:
\begin{equation*}
\begin{tikzcd}[column sep=tiny]
\Hom(\mathbb{K}_{Z_1},\mathbb{K}_{Z_2}) \ar[r] \ar[d] & \Hom(\mathbb{K}_{Z_1}, \mathbb{K}_{\mathbb{CP}^n \times \mathbb{CP}^n \times \mathbb{R}}) \ar[r] \ar[d] & \Hom(\mathbb{K}_{Z_1}, \mathbb{K}_{{Z_2}^c}) \ar[d] \ar[r, "+1"] & \   \\
\Hom(\mathbb{K}_{Z_1}|_W,\mathbb{K}_{Z_2}|_W) \ar[r] \ar[d] & \Hom(\mathbb{K}_{Z_1}|_W, \mathbb{K}_{\mathbb{CP}^n \times \mathbb{CP}^n \times \mathbb{R}}|_W) \ar[r] \ar[d] & \Hom(\mathbb{K}_{Z_1}|_W, \mathbb{K}_{{Z_2}^c}|_W) \ar[d] \ar[r, "+1"] & \   \\ 
\Hom(\mathbb{K}_{Z_1}|_{W'},\mathbb{K}_{Z_2}|_{W'}) \ar[r] \ar[d] & \Hom(\mathbb{K}_{Z_1}|_{W'}, \mathbb{K}_{\mathbb{CP}^n \times \mathbb{CP}^n \times \mathbb{R}}|_{W'}) \ar[r] \ar[d] & \Hom(\mathbb{K}_{Z_1}|_{W'}, \mathbb{K}_{{Z_2}^c}|_{W'}) \ar[d] \ar[r, "+1"] & \   \\
\Hom(\mathbb{K}_{Z_1}|_{W''},\mathbb{K}_{Z_2}|_{W''}) \ar[r] & \Hom(\mathbb{K}_{Z_1}|_{W''}, \mathbb{K}_{\mathbb{CP}^n \times \mathbb{CP}^n \times \mathbb{R}}|_{W''}) \ar[r] & \Hom (\mathbb{K}_{Z_1}|_{W''}, \mathbb{K}_{{Z_2}^c}|_{W''}) \ar[r, "+1"] & \ .  
\end{tikzcd} 
\end{equation*}
The vertical arrows are the restriction maps. By Proposition \ref{prop2.7}, this diagram can be seen as the following commutative diagram:
\begin{equation*}
\hspace*{-2cm}
\begin{tikzcd}[column sep=tiny]
C^*(\mathbb{CP}^n \times \mathbb{CP}^n , \mathbb{CP}^n \times \mathbb{CP}^n \setminus \bigcup_{x \in \mathbb{CP}^n} \{ x \} \times D_x ) \ar[r] \ar[d,"{i_1}^*"] & C^*(\mathbb{CP}^n \times \mathbb{CP}^n) \ar[r] \ar[d] & C^*( \mathbb{CP}^n \times \mathbb{CP}^n \setminus \bigcup_{x \in \mathbb{CP}^n} \{ x \} \times D_x  ) \ar[d] \ar[r, "+1"] & \   \\
C^*( \mathbb{CP}^n , \mathbb{CP}^n \setminus D_{x_0} ) \ar[r] \ar[d,"{i_2}^*"] & C^*( \mathbb{CP}^n) \ar[r] \ar[d] & C^*( \mathbb{CP}^n \setminus D_{x_0} ) \ar[d] \ar[r, "+1"] & \   \\
C^*( D_{x_0,y_0} ,  D_{x_0,y_0} \setminus \{y_0\} ) \ar[r] \ar[d,"{i_3}^*"] & C^*(  D_{x_0,y_0}) \ar[r] \ar[d] & C^*(  D_{x_0,y_0} \setminus \{y_0\}  ) \ar[d] \ar[r, "+1"] & \   \\
C^*( D_{x_0,y_0} \cap B(y_0;\epsilon ) ,   D_{x_0,y_0} \cap B(y_0; \epsilon) \setminus \{y_0\} )  \ar[r] & C^*(   D_{x_0,y_0} \cap B(y_0;\epsilon)    ) \ar[r] & C^*( D_{x_0,y_0} \cap B(y_0;\epsilon) \setminus \{y_0\} ) \ar[r, "+1"] & \ 
\end{tikzcd}
\end{equation*}
where $i_1\colon \mathbb{CP}^n \to \mathbb{CP}^n \times \mathbb{CP}^n$ is the continuous map defined by $i_1(y)=(x_0,y)$, $i_2 \colon  D_{x_0,y_0} \to \mathbb{CP}^n $ is the inclusion map, and  $i_3 \colon   D_{x_0,y_0} \cap  B(y_0;\epsilon) \to D_{x_0,y_0} $ is the inclusion map. Taking its cohomology, we have the following commutative diagram:
\begin{equation*}
\hspace*{-2cm}
\begin{tikzcd}[column sep=tiny]
H^{2}( \mathbb{CP}^n \times \mathbb{CP}^n , \mathbb{CP}^n \times \mathbb{CP}^n \setminus \bigcup_{x \in \mathbb{CP}^n} \{ x \} \times D_x ) \ar[r] \ar[d,"H^2({i_1}^*)"] & H^{2}(\mathbb{CP}^n \times \mathbb{CP}^n) \ar[r] \ar[d] & H^{2}( \mathbb{CP}^n \times \mathbb{CP}^n \setminus \bigcup_{x \in \mathbb{CP}^n} \{ x \} \times D_x  ) \ar[d] \\
H^{2}( \mathbb{CP}^n , \mathbb{CP}^n \setminus D_{x_0} ) \ar[r] \ar[d,"H^2({i_2}^*)"] & H^{2}( \mathbb{CP}^n) \ar[r] \ar[d] & H^{2}( \mathbb{CP}^n \setminus D_{x_0} ) \ar[d] \\
H^{2}( D_{x_0,y_0} , D_{x_0,y_0} \setminus \{y_0\} ) \ar[r] \ar[d,"H^2({i_3}^*)"] & H^{2}(  D_{x_0,y_0}) \ar[r] \ar[d] & H^{2}(  D_{x_0,y_0} \setminus \{y_0\}  ) \ar[d]  \\
H^{2}( D_{x_0,y_0} \cap B(y_0;\epsilon) ,  D_{x_0,y_0} \cap B(y_0;\epsilon) \setminus \{y_0\} )  \ar[r] & H^{2}( D_{x_0,y_0} \cap B(y_0;\epsilon)  ) \ar[r] & H^{2}( D_{x_0,y_0} \cap B(y_0;\epsilon) \setminus \{y_0\} ) .
\end{tikzcd}
\end{equation*}
Recall that $D_{x_0,y_0}$ is diffeomorphic to $\mathbb{CP}^1$. By diagram chasing, it follows that $H^2({i_1}^*)$ and $H^2({i_2}^*)$ are isomorphisms. By the excision theorem, it follows that $H^2({i_3}^*)$ is an isomorphism. Therefore $\psi_1 |_{W''} \in \Hom^0_{\Sh(W'')} ( \mathbb{K}_{Z_1}|_{W''} , \mathbb{K}_{Z_2}[2] |_{W''} )$ is a morphism such that $ H^0(\psi_1 |_{W''}) \in \Ext^2(\mathbb{K}_{Z_1} |_{W''}, \mathbb{K}_{Z_2} |_{W''}) \cong \mathbb{K}$ is a generator. Since the local situation is the same as Lemma \ref{lem3.4}, it follows that $\mathrm{SS} (\Cocone(\psi_1 |_{W''})) \cap \dot{T}^*_{(x_0,y_0,\pi)}{W''} = \{((y_0,\xi) , (\pi, -|\xi|_{g^*})) \in \dot{T}^*_{y_0}{D_{x_0,y_0}} \times T^*_{\pi}\mathbb{R}  \mid (y_0, \xi) \in \dot{T}^*_{y_0}D_{x_0,y_0}  \}$.
\end{proof}

\begin{prop}
\label{prop4.8}
For $i \in \mathbb{Z}_{>0}$ even, there is an equality, 
\begin{equation*}
\mathrm{SS}(\Cocone(\psi_i)) \cap \dot{T}^*_{(x,x,i\pi)}(\mathbb{CP}^n \times \mathbb{CP}^n \times \mathbb{R}) = \Lambda_\phi \cap \dot{T}^*_{(x,x,i\pi)}(\mathbb{CP}^n \times \mathbb{CP}^n \times \mathbb{R}) .
\end{equation*}
\end{prop}
\begin{proof}
We only prove the case of $i=2$ since the proof for $i > 2$ is similar.
Fix a point $x_0 \in \mathbb{CP}^n$.
Set $ W :  =\{x_0\} \times \mathbb{CP}^n \times (2\pi-\epsilon, 2\pi+\epsilon)$ and $ W'  :  = \{x_0\} \times B(x_0;\epsilon) \times (2\pi-\epsilon, 2\pi+\epsilon)$.
Then it suffices to prove the equality $\mathrm{SS} (\Cocone(\psi_2 |_{W'})) \cap \dot{T}^*_{(x_0,x_0,2\pi)}{W'} = \{((x_0,\xi) , (2\pi, -|\xi|_{g^*})) \in \dot{T}^*_{x_0}\mathbb{CP}^n \times T^*_{2\pi}\mathbb{R}  \mid (x_0, \xi) \in \dot{T}^*_{x_0}\mathbb{CP}^n  \}$.
We have the following commuatative diagram:
\begin{equation*}
\begin{tikzcd}[column sep=tiny]
\Hom(\mathbb{K}_{Z_2},\mathbb{K}_{Z_3}) \ar[r] \ar[d] & \Hom(\mathbb{K}_{Z_2}, \mathbb{K}_{\mathbb{CP}^n \times \mathbb{CP}^n \times \mathbb{R}}) \ar[r] \ar[d] & \Hom(\mathbb{K}_{Z_2}, \mathbb{K}_{{Z_3}^c}) \ar[d] \ar[r, "+1"] & \   \\
\Hom(\mathbb{K}_{Z_2}|_W,\mathbb{K}_{Z_3}|_W) \ar[r] \ar[d] & \Hom(\mathbb{K}_{Z_2}|_W, \mathbb{K}_{\mathbb{CP}^n \times \mathbb{CP}^n \times \mathbb{R}}|_W) \ar[r] \ar[d] & \Hom(\mathbb{K}_{Z_2}|_W, \mathbb{K}_{{Z_3}^c}|_W) \ar[d] \ar[r, "+1"] & \   \\
\Hom(\mathbb{K}_{Z_2}|_{W'},\mathbb{K}_{Z_3}|_{W'}) \ar[r] &\Hom(\mathbb{K}_{Z_2}|_{W'}, \mathbb{K}_{\mathbb{CP}^n \times \mathbb{CP}^n \times \mathbb{R}}|_{W'}) \ar[r] & \Hom(\mathbb{K}_{Z_2}|_{W'}, \mathbb{K}_{{Z_3}^c}|_{W'}) \ar[r, "+1"] & \ .
\end{tikzcd} 
\end{equation*}
The vertical arrows are the restriction maps. By Proposition \ref{prop2.7}, this diagram can be seen as the following commutative diagram:
\begin{equation*}
\begin{tikzcd}[column sep=tiny]
C^*(\mathbb{CP}^n \times \mathbb{CP}^n, \mathbb{CP}^n \times \mathbb{CP}^n \setminus \Delta_{\mathbb{CP}^n} ) \ar[r] \ar[d,"{i_1}^*"] & C^*(\mathbb{CP}^n \times \mathbb{CP}^n) \ar[r] \ar[d] & C^*( \mathbb{CP}^n \times \mathbb{CP}^n \setminus \Delta_{\mathbb{CP}^n} ) \ar[d] \ar[r, "+1"] & \   \\
C^*( \mathbb{CP}^n , \mathbb{CP}^n \setminus \{x_0\} ) \ar[r] \ar[d,"{i_2}^*"] & C^*( \mathbb{CP}^n) \ar[r] \ar[d] & C^*( \mathbb{CP}^n \setminus \{x_0\} )  \ar[d] \ar[r, "+1"] & \   \\
C^*( B(x_0 ; \epsilon) , B( x_0; \epsilon) \setminus \{x_0\} ) \ar[r] & C^*( B( x_0; \epsilon) ) \ar[r] & C^*( B( x_0; \epsilon) \setminus \{x_0\} ) \ar[r, "+1"] & \ 
\end{tikzcd}
\end{equation*}
where $i_1 \colon  \mathbb{CP}^n \to \mathbb{CP}^n \times \mathbb{CP}^n$ is the continuous map defined by $i_1(y)=(x_0,y)$ and $i_2 \colon  B(x_0;\epsilon) \to \mathbb{CP}^n$ is the inclusion map. Taking its cohomology, we have the following commutative diagram:
\begin{equation*}
\begin{tikzcd}[column sep=tiny]
{H}^{2n}(\mathbb{CP}^n \times \mathbb{CP}^n , \mathbb{CP}^n \times \mathbb{CP}^n \setminus \Delta_{\mathbb{CP}^n} )  \ar[r] \ar[d,"H^{2n}({i_1}^*)"] & H^{2n}(\mathbb{CP}^n \times \mathbb{CP}^n ) \ar[r] \ar[d] & H^{2n}( \mathbb{CP}^n \times \mathbb{CP}^n \setminus \Delta_{\mathbb{CP}^n} ) \ar[d] \\
{H}^{2n}( \mathbb{CP}^n ,  \mathbb{CP}^n \setminus \{x_0\}) \ar[r] \ar[d,"H^{2n}({i_2}^*)"] & H^{2n}(  \mathbb{CP}^n ) \ar[r] \ar[d] & H^{2n}( \mathbb{CP}^n \setminus \{x_0\} ) \ar[d]  \\
{H}^{2n}( B(x_0 ; \epsilon ) ,  B( x_0; \epsilon ) \setminus \{x_0\} ) \ar[r] & H^{2n}( B(x_0 ; \epsilon)) \ar[r] & H^{2n}( B( x_0; \epsilon) \setminus \{x_0\} ).
\end{tikzcd}
\end{equation*}
By Lemma \ref{lem4.5}, it follows that $H^{2n}({i_1}^*)$ is an isomorphism. By the excision theorem, it follows that $H^{2n}({i_2}^*)$ is an isomorphism. Therefore $\psi_2 |_{W'} \in \Hom^{0}_{\Sh(W')} ( \mathbb{K}_{Z_2}|_{W'} , \mathbb{K}_{Z_3}[2n]|_{W'} ) $ is a morphism such that $ H^0(\psi_2|_{W'}) \in \Ext^{2n}(\mathbb{K}_{Z_2} |_{W'}, \mathbb{K}_{Z_3} |_{W'}) \cong \mathbb{K}$ is a generator. Since the local situation is the same as Lemma \ref{lem3.4}, it follows that $\mathrm{SS} (\Cocone(\psi_2 |_{W'})) \cap \dot{T}^*_{(x_0,x_0,2\pi)}{W'} = \{((x_0,\xi) , (2\pi, -|\xi|_{g^*})) \in \dot{T}^*_{x_0}\mathbb{CP}^n \times T^*_{2\pi}\mathbb{R}  \mid (x_0, \xi) \in \dot{T}^*_{x_0}\mathbb{CP}^n  \}$.
\end{proof}

\newpage

We define the sheaf $K_{+} \in \Sh(\mathbb{CP}^n \times \mathbb{CP}^n \times \mathbb{R})$ as follows. 

For $i \in \mathbb{Z}_{>0}$, we define the sheaf $K_i \in \Sh(\mathbb{CP}^n \times \mathbb{CP}^n \times \mathbb{R})$ by
\begin{equation*}
\begin{split}
K_1& :  =\Cocone(\mathbb{K}_{Z_1} \xrightarrow{\psi_1} \mathbb{K}_{Z_2}[2]), \\
K_2& :  =\Cocone(K_1 \xrightarrow{\tilde{\psi}_2} \mathbb{K}_{Z_3}[2n+1]), \\
K_3& :  =\Cocone(K_2 \xrightarrow{\tilde{\psi}_3} \mathbb{K}_{Z_4}[2n+2]), \\
K_4& :  =\Cocone(K_3 \xrightarrow{\tilde{\psi}_4} \mathbb{K}_{Z_5}[4n+1]), \\
&\dots \\
K_i& :  =\Cocone(K_{i-1} \xrightarrow{\tilde{\psi}_i} \mathbb{K}_{Z_{i+1}}[(i-1)n+2]) \ \text{(if $i$ odd)}, \\
K_i& :  =\Cocone(K_{i-1} \xrightarrow{\tilde{\psi}_i} \mathbb{K}_{Z_{i+1}}[in+1]) \ \text{(if $i$ even)},
\end{split}
\end{equation*}
where $\tilde{\psi}_i $ is the morphism induced by $\psi_i$.
Note that we have the following isomorphisms by induction.
\begin{enumerate}
\item [(i)] the case of $i=2$
\begin{equation*}
\begin{split}
&\Hom_{\Sh(\mathbb{CP}^n \times \mathbb{CP}^n \times \mathbb{R})}(K_1 , \mathbb{K}_{Z_3}[2n+1]) \\
&\cong \Cone(\Hom_{\Sh(\mathbb{CP}^n \times \mathbb{CP}^n \times \mathbb{R})}(\mathbb{K}_{Z_2}[2], \mathbb{K}_{Z_3}[2n+1]) \to \Hom_{\Sh(\mathbb{CP}^n \times \mathbb{CP}^n \times \mathbb{R})}(\mathbb{K}_{Z_1}, \mathbb{K}_{Z_3}[2n+1])) \\
&\cong \Cone( \Hom_{\Sh(\mathbb{CP}^n \times \mathbb{CP}^n \times \mathbb{R})}(\mathbb{K}_{Z_2}[2], \mathbb{K}_{Z_3}[2n+1])\to 0) \\ &\cong \Hom_{\Sh(\mathbb{CP}^n \times \mathbb{CP}^n \times \mathbb{R})}(\mathbb{K}_{Z_2}[2], \mathbb{K}_{Z_3}[2n+1])[1] \\
&\cong \Hom_{\Sh(\mathbb{CP}^n \times \mathbb{CP}^n \times \mathbb{R})}(\mathbb{K}_{Z_2}, \mathbb{K}_{Z_3}[2n]) .
\end{split}
\end{equation*}
\item [(ii)] the case of $i>2$ odd
\begin{equation*}
\begin{split}
&\Hom_{\Sh(\mathbb{CP}^n \times \mathbb{CP}^n \times \mathbb{R})}(K_{i-1} , \mathbb{K}_{Z_{i+1}}[(i-1)n+2]) \\
&\cong \Cone(\Hom_{\Sh(\mathbb{CP}^n \times \mathbb{CP}^n \times \mathbb{R})}(\mathbb{K}_{Z_i}[(i-1)n+1], \mathbb{K}_{Z_{i+1}}[(i-1)n+2]) \\
&\qquad \to \Hom_{\Sh(\mathbb{CP}^n \times \mathbb{CP}^n \times \mathbb{R})}(K_{i-2}, \mathbb{K}_{Z_{i+1}}[(i-1)n+2])) \\
&\cong \Cone( \Hom_{\Sh(\mathbb{CP}^n \times \mathbb{CP}^n \times \mathbb{R})}(\mathbb{K}_{Z_i}[(i-1)n+1], \mathbb{K}_{Z_{i+1}}[(i-1)n+2])\to 0) \\ 
&\cong \Hom_{\Sh(\mathbb{CP}^n \times \mathbb{CP}^n \times \mathbb{R})}(\mathbb{K}_{Z_i}[(i-1)n+1], \mathbb{K}_{Z_{i+1}}[(i-1)n+2])[1] \\
&\cong \Hom_{\Sh(\mathbb{CP}^n \times \mathbb{CP}^n \times \mathbb{R})}(\mathbb{K}_{Z_i}, \mathbb{K}_{Z_{i+1}}[2]) .
\end{split}
\end{equation*}
\item [(iii)] the case of $i>2$ even
\begin{equation*}
\begin{split}
&\Hom_{\Sh(\mathbb{CP}^n \times \mathbb{CP}^n \times \mathbb{R})}(K_{i-1} , \mathbb{K}_{Z_{i+1}}[in+1]) \\
&\cong \Cone(\Hom_{\Sh(\mathbb{CP}^n \times \mathbb{CP}^n \times \mathbb{R})}(\mathbb{K}_{Z_i}[(i-2)n+2], \mathbb{K}_{Z_{i+1}}[in+1]) \\
&\qquad \to \Hom_{\Sh(\mathbb{CP}^n \times \mathbb{CP}^n \times \mathbb{R})}(K_{i-2}, \mathbb{K}_{Z_{i+1}}[in+1])) \\
&\cong \Cone( \Hom_{\Sh(\mathbb{CP}^n \times \mathbb{CP}^n \times \mathbb{R})}(\mathbb{K}_{Z_i}[(i-2)n+2], \mathbb{K}_{Z_{i+1}}[in+1])\to 0) \\ 
&\cong \Hom_{\Sh(\mathbb{CP}^n \times \mathbb{CP}^n \times \mathbb{R})}(\mathbb{K}_{Z_i}[(i-2)n+2], \mathbb{K}_{Z_{i+1}}[in+1])[1] \\
&\cong \Hom_{\Sh(\mathbb{CP}^n \times \mathbb{CP}^n \times \mathbb{R})}(\mathbb{K}_{Z_i}, \mathbb{K}_{Z_{i+1}}[2n]) .
\end{split}
\end{equation*}
\end{enumerate}
Finally we obtain the sequence in $\Sh(\mathbb{CP}^n \times \mathbb{CP}^n \times \mathbb{R})$,
\begin{equation*}
K_1 \leftarrow K_2 \leftarrow K_3 \leftarrow K_4 \leftarrow \cdots .
\end{equation*}
We define $K_{+} :  = \mathrm{holim}_{i>0} {K_i}$. Note that $\mathrm{SS}(K_+)$ and $\Lambda_\phi$ coincide in $\dot{T}^*(\mathbb{CP}^n \times \mathbb{CP}^n \times \mathbb{R}_{>0})$.

In parallel to the construction of $K_{+}$, we define the sheaf $K_{-} \in \Sh(\mathbb{CP}^n \times \mathbb{CP}^n \times \mathbb{R}) $ as follows.

\begin{enumerate}

\item
For $i \in \mathbb{Z}_{<0}$ odd, we define the closed subset $Z_i \subset \mathbb{CP}^n \times \mathbb{CP}^n \times \mathbb{R}$ by
\begin{equation*}
    Z_{i} :  =\{(x,y,t)\in \mathbb{CP}^n \times \mathbb{CP}^n \times \mathbb{R} \mid \mathrm{dist}(x,y)\leq -t+(i+1)\pi \}.
\end{equation*}

For $i \in \mathbb{Z}_{<0}$ even, we define the closed subset $Z_i \subset \mathbb{CP}^n \times \mathbb{CP}^n \times \mathbb{R}$ by
\begin{equation*}
    Z_{i} :  =\{(x,y,t)\in \mathbb{CP}^n \times \mathbb{CP}^n \times \mathbb{R} \mid \mathrm{dist}(x,D_y)\leq -t+(i+1)\pi \}.
\end{equation*}

\item
For $i \in \mathbb{Z}_{<0}$ odd, we choose a morphism $\psi_i \in \Hom^0_{\Sh(\mathbb{CP}^n \times \mathbb{CP}^n \times \mathbb{R})} (\mathbb{K}_{Z_{i-1}} , \mathbb{K}_{Z_i}[2]) $ such that $H^0(\psi_i) \in \Ext^2(\mathbb{K}_{Z_{i-1}}, \mathbb{K}_{Z_i}) \cong \mathbb{K}$ is a generator. For $i \in \mathbb{Z}_{<0}$ even, we choose a morphism $\psi_i \in \Hom^0_{\Sh(\mathbb{CP}^n \times \mathbb{CP}^n \times \mathbb{R})} ( \mathbb{K}_{Z_{i-1}} , \mathbb{K}_{Z_i}[2n]) $ such that $ H^0(\psi_i) \in \Ext^{2n}(\mathbb{K}_{Z_{i-1}}, \mathbb{K}_{Z_{i}} ) \cong \mathbb{K}$ is a generator.

\item
We obtain the sequence in $ \Sh(\mathbb{CP}^n \times \mathbb{CP}^n \times \mathbb{R})$,
\begin{equation*}
\cdots \xrightarrow{\psi_{-4}} \mathbb{K}_{Z_{-4}}[-2n-4] \xrightarrow{\psi_{-3}} \mathbb{K}_{Z_{-3}}[-2n-2] \xrightarrow{\psi_{-2}} \mathbb{K}_{Z_{-2}}[-2] \xrightarrow{\psi_{-1}} \mathbb{K}_{Z_{-1}} .
\end{equation*}

\item
For $i \in \mathbb{Z}_{<0}$, we define the sheaf $K_i \in \Sh(\mathbb{CP}^n \times \mathbb{CP}^n \times \mathbb{R})$ by
\begin{equation*}
\begin{split}
K_{-1}& :  =\Cone(\mathbb{K}_{Z_{-2}}[-2] \xrightarrow{\psi_{-1}} \mathbb{K}_{Z_{-1}}),\\
K_{-2}& :  =\Cone(\mathbb{K}_{Z_{-3}}[-2n-1] \xrightarrow{\tilde{\psi}_{-2}} K_{-1}), \\
K_{-3}& :  =\Cone( \mathbb{K}_{Z_{-4}}[-2n-2] \xrightarrow{\tilde{\psi}_{-3}} K_{-2}), \\
K_{-4}& :  =\Cone(\mathbb{K}_{Z_{-5}}[-4n-1] \xrightarrow{\tilde{\psi}_{-4}} K_{-3}), \\
\dots \\
K_i& :  =\Cone(\mathbb{K}_{Z_{i-1}}[(i+1)n-2] \xrightarrow{\tilde{\psi}_i} K_{{i+1}}) \ \text{(if $i$ odd)}, \\
K_i& :  =\Cone(\mathbb{K}_{Z_{i-1}}[in-1] \xrightarrow{\tilde{\psi}_i} K_{{i+1}}) \ \text{(if $i$ even)},
\end{split}
\end{equation*}
where $\tilde{\psi}_i$ is the morphism induced by $\psi_i$.

\item
We obtain the sequence in $\Sh(\mathbb{CP}^n \times \mathbb{CP}^n \times \mathbb{R})$,
\begin{equation*}
\cdots \leftarrow K_{-4} \leftarrow K_{-3} \leftarrow K_{-2} \leftarrow K_{-1} .
\end{equation*}

\item
We define $K_{-} :  = \mathrm{hocolim}_{i<0}K_i$. Note that $\mathrm{SS}(K_-)$ and $\Lambda_\phi$ coincide in $\dot{T}^*(\mathbb{CP}^n \times \mathbb{CP}^n \times \mathbb{R}_{<0})$.

\end{enumerate}

\begin{lem}
\label{lem4.9}
The object $ \Hom_{\Sh(\mathbb{CP}^n \times \mathbb{CP} \times \mathbb{R})}( K_{-}, K_{+}) $ is isomorphic to $ C^*(\mathbb{CP}^n) \otimes C^*( \mathbb{R}^{2n+1} , S^{2n-1})[-1] $ in $\Mod(\mathbb{K})$.
\end{lem}
\begin{proof}
The proof is the same as Lemma \ref{lem3.7}.
\end{proof}

We choose a morphism $\psi_0 \in \Hom^0_{\Sh(\mathbb{CP}^n \times \mathbb{CP}^n \times \mathbb{R}^n)} ( \mathbb{K}_{-} , \mathbb{K}_{+}[2n+1]) $ such that $H^0(\psi_0) \in \Ext^{2n+1}(K_{-}, K_{+}) \cong \mathbb{K}$ is a generator. Note that the isomorphism $\Ext^{2n+1}(K_{-}, K_{+}) \cong \mathbb{K}$ follows from Lemma \ref{lem4.9}. 

\begin{prop}
\label{prop4.10}
There is an equality,
\begin{equation*}
\mathrm{SS} (\Cocone(\psi_0)) \cap \dot{T}^*_{(x,x,0)} (\mathbb{CP}^n \times \mathbb{CP}^n \times \mathbb{R} ) = \Lambda_\phi \cap \dot{T}^*_{(x,x,0)}(\mathbb{CP}^n \times \mathbb{CP}^n \times \mathbb{R}).
\end{equation*}
\end{prop}
\begin{proof}
It follows from Lemma \ref{lem3.8}.
\end{proof}

\begin{thm}
\label{thm4.11}
The sheaf $K:=\Cocone(\psi_0)$ is the GKS kernel of $\phi$. 
\end{thm}
\begin{proof}
We have to prove that $\mathrm{SS}(K) \setminus 0_{\mathbb{CP}^n \times \mathbb{CP}^n \times \mathbb{R}} = \Lambda_\phi$ and $K|_{ \{ t=0 \} } \cong \mathbb{K}_{\Delta_{\mathbb{CP}^n}}$. The first follows from Proposition \ref{prop4.6}, \ref{prop4.7}, \ref{prop4.8}, and \ref{prop4.10}. The second follows from the following chain of isomorphisms, 
\begin{equation*}
\begin{split}
K|_{\{t=0\}} 
&\cong \Cocone( K_{-}|_{\{t=0\}} \to K_{+}[2n+1]|_{\{t=0\}} ) \\
&\cong \Cocone( K_{-}|_{\{t=0\}} \to 0) \\
&\cong K_{-}|_{\{t=0\}} \\
&\cong K_{-1}|_{\{t=0\}} \\
&\cong \Cone(\mathbb{K}_{Z_{-2}}[-2]|_{\{t=0\}} \to \mathbb{K}_{Z_{-1}}|_{\{t=0\}}) \\
&\cong \Cone(0 \to \mathbb{K}_{Z_{-1}}|_{\{t=0\}}) \\
&\cong \mathbb{K}_{Z_{-1}}|_{\{t=0\}} \\
&\cong \mathbb{K}_{\Delta_{\mathbb{CP}^n}} .
\end{split}
\end{equation*}
\end{proof}

As an application of Theorem \ref{thm4.11}, we solve a conjecture posed in \cite{KuwagakiSaito} for the case of $M=\mathbb{CP}^n$.

\begin{conj}[\cite{KuwagakiSaito}]
\label{conj4.12}
Let $M$ be a compact complex manifold and $F$ be a $\mathbb{C}$-constructible sheaf on $M$. Let $g$ be a Riemannian metric on $M$, $\phi_t\colon \dot{T}^*M \to \dot{T}^*M$ be the time-$t$ geodesic flow induced by $g$, and $\Phi_t\colon \Sh(M) \to \Sh(M)$ be the GKS functor induced by $\phi_t$.
Then there exists an inductive system $\{F_i\}_{i \in I}$ which consists of $\mathbb{C}$-constructible sheaves on $M$ such that $\mathrm{hocolim}_{i \in I} F_i \cong \mathrm{hocolim}_{t \in \mathbb{R}} \Phi_t(F)$.
\end{conj}

\begin{cor}
Conjecture \ref{conj4.12} holds when $M=\mathbb{CP}^n$.
\end{cor}
\begin{proof}
Note that $\Phi_t \simeq {K_\phi}|_{\mathbb{CP}^n \times \mathbb{CP}^n \times \{t\} } \circ-$ where $\circ$ is the composition (\cite{GKS}).
For any $i \in \mathbb{Z}$, the object $K_\phi|_{\{t=2i\pi\}} \in \Sh(\mathbb{CP}^n \times \mathbb{CP}^n)$ is $\mathbb{C}$-constructible because for any $j \in \mathbb{Z}$, the subset $Z_j|_{\{t=2i\pi\}} \subset \mathbb{CP}^n  \times \mathbb{CP}^n$ is either $ \mathbb{CP}^n \times \mathbb{CP}^n $, $\Delta_{\mathbb{CP}^n}$, $ \mathbb{CP}^n \times \mathbb{CP}^n \setminus \Delta_{\mathbb{CP}^n} $, or empty.
Consider the cofinal family $ \{ \Phi_{2i\pi}(F) \}_{ i \in \mathbb{Z}} \subset \{ \Phi_{t}(F) \}_{ t \in \mathbb{R}} $.
Then each object $ \Phi_{2 i \pi }(F) $ is $\mathbb{C}$-constructible because the composition of $\mathbb{C}$-constructible sheaves is again $\mathbb{C}$-constructible.
\end{proof}

\bibliographystyle{alpha}
\bibliography{bibs.bib}

\end{document}